\documentclass[a4paper,10pt,twoside]{article}

\usepackage{tensor}
\usepackage[utf8]{inputenc}
\usepackage[UKenglish]{babel}
\usepackage{cite}
\usepackage{amssymb}
\usepackage{amsmath}
\usepackage{mathrsfs}
\usepackage{amsthm}
\usepackage[T1]{fontenc}
\usepackage{accents}
\usepackage[cm]{fullpage}
\usepackage{graphicx}
\usepackage{tensor}
\usepackage{relsize}
\usepackage{appendix}
\usepackage[affil-it]{authblk}

\RequirePackage{lastpage}
\RequirePackage{latexsym}
\makeatletter
\ifx\@NODS\undefined\RequirePackage{dsfont}\fi
\ifx\@NOAMS\undefined\RequirePackage{amsmath,amsfonts,amssymb,amsthm}\fi
\makeatother

\RequirePackage{geometry}
\RequirePackage{bera}
\RequirePackage[pdftex,pagebackref=false]{hyperref}
\hypersetup{pdfborder=0 0 0}
\hypersetup{pdfstartview={FitH}}

\title{{\Large \bfseries{Brownian motion and the distance to a submanifold}}}

\author{James Thompson%
  \footnote{University of Warwick, Email: \texttt{j.thompson.2@warwick.ac.uk}}}

\date{\today}

\makeatletter
\def\@MRExtract#1 #2!{#1}
\newcommand{\MR}[1]{
  \xdef\@MRSTRIP{\@MRExtract#1 !}%
  \href{http://www.ams.org/mathscinet-getitem?mr=\@MRSTRIP}{MR-\@MRSTRIP}}
\makeatother

\makeatletter
\renewenvironment{thebibliography}[1]{%
  \section*{\refname
    \@mkboth{\MakeUppercase\refname}{\MakeUppercase\refname}}%
  \phantomsection%
  \addcontentsline{toc}{section}{\refname}%
  \list{\@biblabel{\@arabic\c@enumiv}}%
  {\settowidth\labelwidth{\@biblabel{#1}}%
    \small%
    \setlength{\labelsep}{0.4em}%
    \setlength{\leftmargin}{\labelwidth}%
    \addtolength{\leftmargin}{\labelsep}%
    \setlength{\itemsep}{-.25em}%
    \@openbib@code
    \usecounter{enumiv}%
    \let\p@enumiv\@empty
    \renewcommand\theenumiv{\@arabic\c@enumiv}}%
  \sloppy\clubpenalty4000\@clubpenalty\clubpenalty\widowpenalty4000%
  \sfcode`\.\@m}{%
  \def\@noitemerr{%
    \@latex@warning{Empty `thebibliography' environment}}%
  \endlist}
\makeatother

\makeatletter
\ifx\@NODS\undefined%

\let\mathbb=\mathds
\else%
\fi
\makeatother

\DeclareMathOperator{\vol}{vol}
\DeclareMathOperator{\Ric}{Ric}

\DeclareMathOperator{\Cut}{Cut}
\DeclareMathOperator{\Ent}{Ent}
\DeclareMathOperator{\scal}{scal}

\newcommand{\dimm}{m}
\newcommand{\dimn}{n}
\newcommand{\sv}{k}
\newcommand{\mom}{p}
\newcommand{\secur}{\kappa}

\newtheorem{theorem}{Theorem}[section]
\newtheorem{lemma}[theorem]{Lemma}
\newtheorem{proposition}[theorem]{Proposition}
\newtheorem{corollary}[theorem]{Corollary}

\newtheorem{remark}[theorem]{Remark}

\newtheorem{example}[theorem]{Example}

\begin{document}

\maketitle

\begin{abstract}%
   \noindent%
   {We present a study of the distance between a Brownian motion and a submanifold of a complete Riemannian manifold. We include a variety of results, including an inequality for the Laplacian of the distance function derived from a Jacobian comparison theorem, a characterization of local time on a hypersurface which includes a formula for the mean local time, an exit time estimate for tubular neighbourhoods and a concentration inequality. We derive the concentration inequality using moment estimates to obtain an exponential bound, which holds under fairly general assumptions and which is sufficiently sharp to imply a comparison theorem. We provide numerous examples throughout. Further applications will feature in a subsequent article, where we see how the main results and methods presented here can be applied to certain study objects which appear naturally in the theory of \textit{submanifold bridge processes}.}\\[1em]%
   {\footnotesize%
     \par%
     \par
   }%footnotesize
 \end{abstract}

\section*{Introduction}

Suppose that $M$ is a complete and connected Riemannian manifold of dimension $\dimm$, that $N$ is a closed embedded submanifold of $M$ of dimension $\dimn\in\lbrace 0,\ldots,\dimm-1\rbrace$ and that $X(x)$ is a Brownian motion on $M$ starting at $x\in M$ with explosion time $\zeta(x)$. The main objective of this paper is to study the distance between $X_t(x)$ and $N$ for each $t\geq 0$. This is not something which has previously been considered in the literature, the closest reference being the study of mean exit times given by Gray, Karp and Pinksy in \cite{MR830234}, and our main results are Theorems \ref{th:estimatefinal}, \ref{th:expenbm} and \ref{th:concenthm}. We denote by $r_N$ the distance function and assume that there exist constants $\nu \geq 1$ and $\lambda \in \mathbb{R}$ such that the Lyapunov-like condition
\begin{equation}\label{eq:lapineqintro}
\frac{1}{2}\triangle r_N^2 \leq \nu + \lambda r_N^2
\end{equation}
holds off the cut locus of $N$. Geometric conditions under which such an inequality arises, which allow for unbounded curvature, are given by Theorem \ref{th:estimatefinal} (see Corollary \ref{cor:maincor}), which is derived from the classical Heintze-Karcher comparison theorem \cite{MR533065}. Under such assumptions we deduce a variety of probabilistic estimates, which are presented in Section \ref{se:section3}. We do so first using a logarithmic Sobolev inequality. Such inequalities were originally studied by Gross in \cite{MR0420249} and we refer to the article \cite{MR1453132} for the special case of the heat kernel measure. We then prove more general estimates using the It\^{o}-Tanaka formula of Section \ref{se:section2}, which is derived from the formula of Barden and Le \cite{MR1314177} and which reduces to the formula of Cranston, Kendall and March \cite{MR1231929} in the one-point case. The basic method is similar to that of Hu in \cite{MR1694764}, who studied (uniform) exponential integrability for diffusions in $\mathbb{R}^\dimm$ for $C^2$ functions satisfying another Lyapunov-like condition. Indeed, several of our results could just as well be obtained for such functions, but we choose to focus on the distance function since in this case we have a geometric interpretation for the Laplacian inequality \eqref{eq:lapineqintro}. Section \ref{se:section2} also includes a characterization of the local time of Brownian motion on a hypersurface and a couple of examples. The probabilistic estimates of Section \ref{se:section3} include Theorem \ref{th:secradmomthm} and its generalization Theorem \ref{th:evenradmomthm}, which provide upper bounds on the even moments of the distance $r_N(X_t(x))$ for each $t\geq 0$. Using properties of Laguerre polynomials we use these estimates to deduce Theorems \ref{th:inequality2} and \ref{th:expenbm}, which provide upper bounds on the moment generating functions of the distance $r_N(X_t(x))$ and squared distance $r_N^2(X_t(x))$ for each $t\geq 0$, respectively. The latter estimate improves and generalizes a theorem of Stroock. Using Theorem \ref{th:expenbm} and Markov's inequality we then deduce Theorem \ref{th:concenthm}, which provides an concentration inequality for tubular neighbourhoods. Note that in this paper we do not assume the existence of Gaussian upper bounds for the heat kernel; the constants appearing in our estimates are all explicit. This paper has been presented in such a way that it should be possible for the reader to read Section \ref{se:section3} either before or after Sections \ref{se:section1} and \ref{se:section2}. Finally, the author wishes to thank his supervisor Xue-Mei Li for supporting this piece of work with numerous insightful suggestions. He also wishes to thank David Elworthy and an anonymous referee for providing additional comments and criticism.

\section{Geometric Inequalities}\label{se:section1}

We begin by deriving an inequality for the Laplacian of the distance function. This object can be written in terms of the Jacobian determinant of the normal exponential map, so we begin with a comparison theorem and inequalities for this object. The main references here are \cite{MR533065} and \cite{MR1390760}.

\subsection{Heintze-Karcher Inequalities}\label{ss:hkest}

Suppose that $M$ is a complete and connected Riemannian manifold of dimension $\dimm$ with Riemannian volume measure $\vol_M$ and that $N$ is a closed embedded submanifold of $M$ of dimension $\dimn \in \lbrace 0,\ldots,\dimm-1\rbrace$ with induced measure $\vol_N$ and normal bundle $\pi_N :TN^\bot\rightarrow N$. Fix $\xi\in UTN^\bot:=\lbrace \xi \in TN^\bot : \|\xi \|=1\rbrace$ and choose $t_1 \in (0,f_N (\xi))$ where $f_N(\xi)$ denotes the first focal time along the geodesic $\gamma_\xi$ with $\gamma_\xi(0)=\pi_N (\xi)$ and $\dot \gamma_\xi(0)=\xi$. If $p \in M$ with $\sigma_p$ a two-dimensional subspace of $T_p M$ then denote by $K(\sigma_p)$ the sectional curvature of $\sigma_p$ and let
\begin{equation*}
\underline{\secur}_{\xi}(t_1):= \min \left\{   
  K(\sigma_{\gamma_\xi (t)})\colon \begin{aligned} & \sigma_{\gamma_\xi (t)} \text{ is any two-dimensional subspace of}\\ 
  &T_{\gamma_\xi (t)} M \text{ containing } \dot{\gamma_\xi} (t) \text{ for any } t\in \left[0,t_1\right]
  \end{aligned}
\right\}.
\end{equation*}
As regards the extrinsic geometry of the submanifold, denote by $A_\xi$ the shape operator associated to $\xi$ and by $\lambda_1(\xi),\ldots,\lambda_\dimn(\xi)$ the eigenvalues of $A_\xi$. These eigenvalues are called the \textit{principle curvatures of $N$ with respect to $\xi$} and their arithmetic mean, denoted by $H_\xi$, is called the \textit{mean curvature of $N$ with respect to $\xi$}. Let $\theta_N:TN^{\bot}\rightarrow \mathbb{R}$ denote the Jacobian determinant of the normal exponential map $\exp_N:TN^\bot \rightarrow TM$ (which is simply the exponential map of $M$ restricted to $TN^\bot$) and for $\kappa,\lambda \in \mathbb{R}$ define, for comparison, functions $S_\kappa$, $C_\kappa$, $G_{\kappa}$ and $F_{\kappa}^{\lambda}$ by
\begin{equation*}
\begin{split}
S_\kappa (t) &:=
\begin{cases}
\frac{1}{\sqrt{\kappa}} \sin \sqrt{\kappa}t    &\mbox{if } \kappa > 0 \\
					   t    &\mbox{if } \kappa = 0 \\
\frac{1}{\sqrt{-\kappa}} \sinh \sqrt{-\kappa}t &\mbox{if } \kappa < 0 \\
\end{cases}
\end{split}
\end{equation*}
\begin{equation*}
\begin{split}
C_\kappa(t) &:= \frac{d}{dt} S_\kappa(t)\\[0.2cm]
G_{\kappa}(t)&:= \frac{d}{dt}\log( S_\kappa(t)/t)\\[0.2cm]
F^\lambda_\kappa(t)&:= \frac{d}{dt}\log (C_{\kappa} (t) + \lambda S_{\kappa} (t)).
\end{split}
\end{equation*}
If $\secur_{\xi}(t_1)$ is \textit{any} constant such that $\secur_{\xi}(t_1) \leq \underline{\secur}_{\xi}(t_1)$ then, as was proved by Heintze and Karcher in \cite{MR533065} using a comparison theorem for Jacobi fields, and which was generalized somewhat by Kasue in \cite{MR722530}, there is the inequality
\begin{equation}\label{eq:logestimate}
\frac{d}{dt} \log \theta_N (t \xi) \leq (\dimm-\dimn-1)G_{\secur_{\xi}(t_1)} (t)+ \sum_{i=1}^{\dimn} F_{\secur_{\xi}(t_1)}^{\lambda_i(\xi)}(t)
\end{equation}
for all $0\leq t \leq t_1$. Furthermore, if $\underline{\rho}_{\xi}(t_1)$ satisfies
\begin{equation*}
(\dimm-1)\underline{\rho}_{\xi}(t_1)=\min\lbrace \Ric(\dot \gamma_\xi(t)):0\leq t\leq t_1\rbrace 
\end{equation*}
and $\rho_{\xi}(t_1)$ is \textit{any} constant such that $\rho_{\xi}(t_1) \leq \underline{\rho}_{\xi}(t_1)$ then for $\dimn=0$ there is the inequality
\begin{equation}\label{eq:logestimate2}
\frac{d}{dt} \log \theta_N (t \xi) \leq (\dimm-1)G_{\rho_{\xi}(t_1)} (t)
\end{equation}
for all $0\leq t \leq t_1$ and for $\dimn=\dimm-1$ there is the inequality
\begin{equation}\label{eq:logestimate3}
\frac{d}{dt} \log \theta_N (t \xi) \leq (\dimm-1) F_{\rho_{\xi}(t_1)}^{H_\xi}(t)
\end{equation}
for all $0\leq t \leq t_1$. Heintze and Karcher's method implies that the right-hand sides of inequalities \eqref{eq:logestimate}, \eqref{eq:logestimate2} and \eqref{eq:logestimate3} are finite for all $0\leq t \leq t_1$. Note that the `empty sum is zero' convention is used to cover the case $\dimn=0$ in inequality \eqref{eq:logestimate}.

\subsection{Jacobian Inequalities}\label{ss:jacineq}

We can use the Heintze-Karcher inequalities to deduce secondary estimates.
\begin{proposition}\label{prop:propos}
For $0\leq t\leq t_1$ we have
\begin{equation*}
\frac{d}{dt} \log \theta_N (t \xi) \leq  (\dimm-1) \sqrt{\vert{\underline{\secur}_\xi(t_1)\wedge 0}\vert} + \sum_{i=1}^{\dimn} \vert \lambda_i(\xi) \vert.
\end{equation*}
\end{proposition}
By Gronwall's inequality and the fact that $\theta_N\vert_N\equiv 1$, this differential inequality implies an upper bound on $\theta_N$. By a change of variables, upper bounds on $\theta_N$ imply upper bounds on the volumes of tubular neighbourhoods and the areas of the boundaries. In this way we can obtain estimates on these objects which are more explicit than those found in \cite{MR2024928}. We will not present these calculations here; they can be found in \cite{THESIS}. To prove the proposition we will use two preliminary lemmas.
\begin{lemma}\label{lem:Case1}
If ${\underline{\secur}_\xi(t_1)} \geq 0$ then for $0\leq t \leq t_1$ we have
\begin{equation*}
\frac{d}{dt} \log \theta_N (t\xi) \leq \sum_{i=1}^\dimn  \lambda_i(\xi).
\end{equation*}
\end{lemma}
\begin{proof}
Setting $\secur_{\xi}(t_1)=0$ we have $S_{\secur_\xi(t_1)} (t) = t$ and $C_{\secur_\xi(t_1)} (t) = 1$ and from inequality \eqref{eq:logestimate} it follows that
\begin{equation*}
\frac{d}{dt} \log \theta_N (t\xi) \leq \sum_{i=1}^{\dimn} \frac{\lambda_i(\xi)}{1+\lambda_i(\xi) t}
\end{equation*}
for all $0\leq t \leq t_1$. The result follows by considering the cases $\lambda_i(\xi) \geq 0$ and $\lambda_i(\xi) < 0$ separately.
\end{proof}
\begin{lemma}\label{lem:Case3}
If ${\underline{\secur}_\xi(t_1)} < 0$ then for $0\leq t \leq t_1$ we have
\begin{align}
\frac{d}{dt} \log \theta_N (t \xi) &\leq (\dimm-\dimn-1)\sqrt{-{\underline{\secur}_\xi(t_1)}}\nonumber \\
&\quad\quad+\sum_{i=1}^{\dimn} \left( \sqrt{-{\underline{\secur}_\xi(t_1)}} \mathbb{1}_{\{ \vert \lambda_i(\xi) \vert < \sqrt{-{\underline{\secur}_\xi(t_1)}}\}}+\lambda_i(\xi)\mathbb{1}_{\{\vert \lambda_i(\xi) \vert \geq \sqrt{-{\underline{\secur}_\xi(t_1)}} \}} \right).\nonumber
\end{align}
\end{lemma}
\begin{proof}
Fix $\kappa <0$ and $\lambda \in \mathbb{R}$. Note that $\lim_{t \downarrow 0} \left(\coth(t)-1/t\right) = 0$, $\lim_{t \uparrow \infty} \left(\coth(t)-1/t\right) = 1$ and that by Taylor's theorem the derivative of this function is strictly positive for positive $t$. Therefore $\coth(t)-1/t \leq 1$ for $t \in (0,\infty)$ and $G_{\kappa} (t) \leq \sqrt{-\kappa}$. Note that for each $i=1,\ldots,k$ we have
\begin{equation*}
-\frac{d}{dt} F_{\kappa}^{\lambda} (t) = \frac{{\kappa}+ \lambda^2}{(C_{\kappa} (t) + \lambda S_{\kappa} (t))^2}
\end{equation*}
so $F_{\kappa}^{\lambda} $ is increasing on $\left( 0, t_1 \right]$ if and only if $\vert \lambda \vert < \sqrt{-\kappa}$. If $\vert \lambda \vert \geq \sqrt{-\kappa}$ then $F^{\kappa}_{\lambda}$ is nonincreasing and $F^{\kappa}_{\lambda} (t) \leq \lim_{t \downarrow 0} F^{\kappa}_{\lambda} (t) = \lambda$. Conversely if $\vert \lambda \vert < \sqrt{-\kappa}$ then
\begin{equation*}
C_{\kappa} (t) +\lambda  S_{\kappa} (t) \geq \cosh \left(\sqrt{-\kappa}t\right)- \sinh \left(\sqrt{-\kappa}t\right) = e^{-\sqrt{-{\kappa}}t}
\end{equation*}
so $F^{\kappa}_{\lambda} $ is defined on $(0, \infty)$ and
\begin{equation*}
F^{\kappa}_{\lambda } (t) \leq \lim_{t \uparrow \infty} F^{\kappa}_{\lambda } (t) \leq \sqrt{-{\kappa}} \lim_{t \uparrow \infty} \left( \frac{\sinh(t)+\cosh(t)}{\cosh(t)-\sinh(t)}\right) = \sqrt{-{\kappa}}.
\end{equation*}
The lemma then follows from inequality \eqref{eq:logestimate} by setting $\secur_{\xi}(t_1)=\underline{\secur}_\xi(t_1)$.
\end{proof}
We can now prove Propostion \ref{prop:propos}.
\begin{proof}[Proof of Proposition \ref{prop:propos}]
By Lemmas \ref{lem:Case1} and \ref{lem:Case3} it follows that
\begin{eqnarray}
& & \frac{d}{dt} \log \theta_N (t\xi) \nonumber \\[0.2cm]
&\leq&   \sum_{i=1}^\dimn \lambda_i(\xi) \mathbb{1}_{\{{\underline{\secur}_\xi(t_1)} \geq 0 \}} + (\dimm-\dimn-1)\sqrt{-{\underline{\secur}_\xi(t_1)}} \mathbb{1}_{\{{\underline{\secur}_\xi(t_1)} < 0\}}  \nonumber \\
& & + \sum_{i=1}^{\dimn} \left( \sqrt{-\underline{\secur}_\xi(t_1)}\mathbb{1}_{\{ \vert \lambda_i(\xi) \vert < \sqrt{-\underline{\secur}_\xi(t_1)}\}}  + \lambda_i(\xi) \mathbb{1}_{\{\vert \lambda_i(\xi) \vert \geq \sqrt{-{\underline{\secur}_\xi(t_1)}} \}} \right) \mathbb{1}_{\{{\underline{\secur}_\xi(t_1)} < 0\}} \nonumber \\[0.2cm]
&\leq&   \sum_{i=1}^\dimn \vert\lambda_i(\xi)\vert \mathbb{1}_{\{{\underline{\secur}_\xi(t_1)} \geq 0 \}} + (\dimm-\dimn-1)\sqrt{\vert{\underline{\secur}_\xi(t_1) \wedge 0}\vert}  \nonumber \\
& & + \dimn \sqrt{\vert{\underline{\secur}_\xi(t_1) \wedge 0}\vert} + \sum_{i=1}^{\dimn} \vert\lambda_i(\xi)\vert \mathbb{1}_{\{{\underline{\secur}_\xi(t_1)} < 0\}} \nonumber \\[0.2cm]
&=& (\dimm-1)\sqrt{\vert{\underline{\secur}_\xi(t_1) \wedge 0}\vert} +\sum_{i=1}^\dimn \vert \lambda_i(\xi)\vert \nonumber
\end{eqnarray}
as required.
\end{proof}
Note that the factor $(\dimm-1)$ is reasonable since an orthonormal basis of a tangent space $T_{\gamma_{\xi}}M$ gives rise to precisely $(\dimm-1)$ orthogonal planes containing the radial direction $\dot{\gamma_\xi}$.

\subsection{Laplacian Inequalities}\label{ss:lapineq}

Denote by $\mathcal{M}(N)$ the largest domain in $TN^\bot$ whose fibres are star-like and such that $\exp_N \vert_{\mathcal{M}(N)}$ is a diffeomorphism onto its image. Then that image is $M\setminus \Cut(N)$, where $\Cut(N)$ denotes the cut locus of $N$. Recall that $\Cut(N)$ is a closed subset of $M$ with $\vol_M$-measure zero. With $r_N:M\rightarrow \mathbb{R}$ defined by $r_N(\cdot) := d_M(\cdot,N)$ the vector field $\frac{\partial}{\partial r_N}$ will denote differentiation in the radial direction, which is defined off the union of $N$ and $\Cut(N)$ to be the gradient of $r_N$ and which is set equal to zero on that union. If as in \cite[p.146]{MR2024928} we define a function $\Theta_N : M\setminus \Cut(N) \rightarrow \mathbb{R}$ by
\begin{equation}\label{eq:Thetadefn}
\Theta_N :=\theta_N \circ {\left(\exp_N \vert_{\mathcal{M}(N)}\right)}^{-1}
\end{equation}
we then have the following corollary of Proposition \ref{prop:propos}, in which $c_N(\xi)$ denotes the distance to the cut locus along $\gamma_\xi$.
\begin{corollary}\label{cor:corone}
Suppose that there is a function $\secur : \left[0,\infty\right) \rightarrow \mathbb{R}$ such that for each $\xi \in UTN^\bot$ and $t_1 \in \left(0,c_N (\xi)\right)$ we have $\secur(t_1)\leq \underline{\secur}_\xi(t_1)$. Furthermore, suppose that the principal curvatures of $N$ are bounded in modulus by a constant $\Lambda\geq 0$. Then there is the estimate
\begin{equation}\label{eq:estinkappa1}
\frac{\partial}{\partial r_N}\log \Theta_N \leq \dimn \Lambda + (\dimm-1)\sqrt{\vert \secur(r_N)\wedge 0 \vert}
\end{equation}
on $M\setminus \Cut(N)$.
\end{corollary}
\begin{proof}
For each $\xi \in UTN^{\bot}$ and $t_1 \in \left(0,c_N (\xi)\right)$ we see by Proposition \ref{prop:propos} that
\begin{equation*}
\frac{\partial}{\partial r_N}\log \Theta_N (\gamma_\xi (t_1))=\frac{d}{d t}\log \theta_N(t\xi)\bigg\vert_{t=t_1} \leq \dimn \Lambda+(\dimm-1)\sqrt{\vert \secur(t_1)\wedge 0 \vert}.
\end{equation*}
Since for each $p \in M\setminus (N\cup \Cut(N))$ there exists a unique $\xi_p \in UTN^{\bot}$ such that $\gamma_{\xi_p}(r_N(p))= p$, the result follows for such $p$ by setting $t_1 = r_N(p)$. For $p\in N$ the radial derivative is set equal to zero in which case the result is trivial.
\end{proof}
Furthermore, following from remarks made at the end of Subsection \ref{ss:hkest}, if there is a function $\rho : \left[0,\infty\right) \rightarrow \mathbb{R}$ such that for each $\xi \in UTN^\bot$ and $t_1 \in \left(0,c_N (\xi)\right)$ we have $ \rho(t_1)\leq \underline{\rho}_\xi(t_1)$ then for $\dimn=0$ there is the estimate
\begin{equation}\label{eq:estinkappa2}
\frac{\partial}{\partial r_N}\log \Theta_N \leq (\dimm-1)\sqrt{\vert \rho(r_N)\wedge 0 \vert}
\end{equation}
on $M \setminus \Cut(N)$ and for $\dimn=\dimm-1$ with $\vert H_\xi \vert \leq \Lambda$ for each $\xi \in UTN^\bot$ there is the estimate
\begin{equation}\label{eq:estinkappa3}
\frac{\partial}{\partial r_N}\log \Theta_N \leq (\dimm-1)\left(\sqrt{\vert \rho(r_N)\wedge 0 \vert}+\Lambda\right)
\end{equation}
on $M \setminus \Cut(N)$. Thus we arrive at the main results of this section, which are the following theorem and its corollary.
\begin{theorem}\label{th:estimatefinal}
Suppose that $M$ is a complete and connected Riemannian manifold of dimension $\dimm$ and that $N$ is a closed embedded submanifold of $M$ of dimension $\dimn \in \lbrace 0,\ldots,\dimm-1\rbrace$. Denote by $\Cut(N)$ the cut locus of $N$ and $r_N$ the distance to $N$ and suppose that there exist constants $C_1,C_2 \geq 0$ such that one of the following conditions is satisfied on $M \setminus (\Cut(N)\cup N)$:
\begin{description}
\item[(C1)] the sectional curvatures of planes containing the radial direction are bounded below by $-(C_1 + C_2 r_N)^2$ and there exists a constant $\Lambda \geq 0$ such that the principal curvatures of $N$ are bounded in modulus by $\Lambda$;
\item[(C2)] $\dimn=0$ and the Ricci curvature in the radial direction is bounded below by $-(\dimm-1)(C_1 + C_2 r_N)^2$;
\item[(C3)] $\dimn=\dimm-1$ and the Ricci curvature in the radial direction is bounded below by $-(\dimm-1)(C_1 + C_2 r_N)^2$ and there exists a constant $\Lambda\geq 0$ such that the mean curvature of $N$ is bounded in modulus by $\Lambda$.
\end{description}
Then for $\Theta_N$, defined by \eqref{eq:Thetadefn}, we have the inequality
\begin{equation*}
\frac{\partial}{\partial r_N}\log \Theta_N \leq \dimn \Lambda + (\dimm-1)(C_1+C_2r_N)
\end{equation*}
where $\frac{\partial}{\partial r_N}$ denotes differentiation in the radial direction.
\end{theorem}
Note that if $\dimn=0$ then the mean curvature is not relevant and if $\dimm=1$ then the sectional curvatures are not relevant but that the above estimates still make sense in either of these cases. Recall that we are primarily interested in the Laplacian of the distance function. If $\triangle$ denotes the Laplace-Beltrami operator on $M$ then, as shown in \cite{MR2024928}, there is the formula
\begin{equation*}
\frac{1}{2}\triangle r_N^2 = (\dimm-\dimn) +r_N \frac{\partial}{\partial r_N} \log \Theta_N
\end{equation*}
on $M \setminus \Cut(N)$. This yields the following corollary.
\begin{corollary}\label{cor:maincor}
Under the conditions of Theorem \ref{th:estimatefinal} we have
\begin{equation}\label{eq:simpleestimate1}
\frac{1}{2}\triangle r_N^2 \leq (\dimm-\dimn) + (\dimn \Lambda + (\dimm-1)C_1)r_N +(\dimm-1)C_2r_N^2
\end{equation}
on $M\setminus \Cut(N)$.
\end{corollary}
For the particular case in which $N$ is a point $p$, it was proved by Yau in \cite{MR0417452} that if the Ricci curvature is bounded below by a constant $R$ then the Laplacian of the distance function $r_p$ is bounded above by $(\dimm-1)/r_p$ plus a constant depending on $R$. Yau then used analytic techniques in \cite{MR505904} to prove that this bound implies the stochastic completeness of $M$. A relaxation of Yau's condition which allows the curvature to grow like a negative quadratic in the distance function is essentially optimal from the point of view of curvature and nonexplosion; this is why we did not feel it necessary to present Theorem \ref{th:estimatefinal} in terms of a general growth function, although one certainly could. We will return to this matter in Subsection \ref{ss:exittimes}. If in Yau's example we set $(\dimm-1)\varrho = R$ then inequality \eqref{eq:logestimate2} and Taylor expansions imply $G^\varrho(t)\leq -\varrho t/3$ for all $t\geq 0$ if $\varrho \leq 0$ or for $t \in [0,\frac{\pi}{\sqrt{\varrho}})$ if $\varrho >0$, which yields the simple estimate
\begin{equation*}
\frac{\partial}{\partial r_p}\log \Theta_p \leq -\frac{Rr_p}{3}
\end{equation*}
on $M\setminus \Cut(p)$. This has the advantage of taking into account the effect of positive curvature and in turn yields the Laplacian estimate
\begin{equation}\label{eq:simpleestimate1R}
\frac{1}{2}\triangle r_p^2 \leq \dimm -\frac{Rr_p^2}{3}
\end{equation}
on $M\setminus \Cut(p)$, which is different to Yau's bound. Note that inequalities \eqref{eq:simpleestimate1} and \eqref{eq:simpleestimate1R} actually hold on the whole of $M$ in the sense of distributions.

\section{Local Time}\label{se:section2}

In this section we show how the distance function relates to the local time of Brownian motion on a hypersurface. Since the boundaries of regular domains are included as examples, this could yield applications related to the study of reflected Brownian motion. The main references here are \cite{MR1314177}, \cite{MR1489143} and \cite{MR1231929}. The articles \cite{MR1314177} and \cite{MR1489143} approach geometric local time in the general context of continuous semimartingales from the point of view of Tanaka's formula while \cite{MR1231929} approaches the topic for the special case of Brownian motion using Markov process theory.

\subsection{It\^{o}-Tanaka Formula}\label{ss:itoformula}

Suppose that $X$ is a Brownian motion on $M$ defined up to an explosion time $\zeta$ and that $U$ is a horizontal lift of $X$ to the orthonormal frame bundle with antidevelopment $B$ on $\mathbb{R}^\dimm$. In \cite{MR1941909} it was proved, using the theory of viscosity solutions, that the cut locus of $N$ is given by the disjoint union of two subsets $\mathring{C}(N)$ and $\check{C}(N)$ where the connected components of $\mathring{C}(N)$, of which there are at most countably many, are smooth two-sided $(\dimm-1)$-dimensional submanifolds of $M$ and where $\check{C}(N)$ is a closed subset of $M$ of Hausdorff dimension at most $\dimm-2$ (and therefore polar for $X$ by \cite{MR771826}). Also $\mathring{C}(N)\cup N$ has $\vol_M$-measure zero so it follows that
\begin{equation*}
\int_0^\cdot \bigg\langle \frac{\partial}{\partial r_N},U_s dB_s\bigg\rangle= \beta_\cdot
\end{equation*}
where $\beta$ is a standard one-dimensional Brownian motion, by L\'{e}vy's characterization and the fact that $U$ consists of isometries. Furthermore, points belonging to $\mathring{C}(N)$ can be connected to $N$ by precisely \textit{two} length-minimizing geodesic segments, both of which are non-focal. Using these observations, it follows from \cite[Theorem~1]{MR1314177} that $r_N(X)$ is a continuous semimartingale. In particular, if $\tau$ is a stopping time with $0\leq\tau<\zeta$ then there exist continuous adapted nondecreasing and nonnegative processes $L^{N}(X)$ and $L^{\mathring{C}(N)}(X)$, whose associated random measures are singular with respect to Lebesgue measure and supported when $X$ takes values in $N$ and $\mathring{C}(N)$, respectively, such that
\begin{equation}\label{eq:rNFormulaforBM}
{r_N}(X_{t\wedge \tau}) = {r_N}(X_0) + \beta_{t\wedge \tau} + \frac{1}{2}\int_0^{t\wedge \tau}\triangle r_N(X_s) ds - \mathbb{L}^{\Cut(N)}_{t\wedge \tau}(X) + L^N_{t\wedge \tau}(X)
\end{equation}
for all $t\geq 0$, almost surely, where
\begin{equation*}
d\mathbb{L}^{\Cut(N)}(X) :=\frac{1}{2}\left( D^-_X - D^+_X\right) {r_N}(\mathbf{n}) dL^{\mathring{C}(N)}(X).
\end{equation*}
Here $\mathbf{n}$ is \textit{any} unit normal vector field on $\mathring{C}(N)$ and the G\^{a}teaux derivatives $D^\pm {r_N}$ are defined for $z \in \mathring{C}(N)$ and $v \in T_zM$ by
\begin{equation*}
D^+_z {r_N}(v) := \lim_{\epsilon \downarrow 0} \frac{1}{\epsilon} \left(f(\exp_z (\epsilon v)) - {r_N}(z)\right) 
\end{equation*}
and $D^-_z {r_N}(v) := -D^+_z {r_N}(-v)$. A detailed explanation of precisely how formula \eqref{eq:rNFormulaforBM} is derived from \cite[Theorem~1]{MR1314177} can be found in the author's doctoral thesis. Note that the integral appearing in formula (\ref{eq:rNFormulaforBM}) is well-defined since the set of times when $X \in N \cup \Cut(N)$ has Lebesgue measure zero. The process $L^{\mathring{C}(N)}(X)$ is given by the local time of $d(X,\mathring{C}(N))$ at zero so long as the latter makes sense, while for the one-point case the process $\mathbb{L}^{\Cut(N)}(X)$ coincides with the geometric local time introduced in \cite{MR1231929}. The process $L^{N}(X)$, which we will refer to as \textit{the local time of $X$ on $N$}, satisfies
\begin{equation*}
L^{N}(X)=L^0(r_N(X))
\end{equation*}
where $L^0(r_N(X))$ denotes the (symmetric) local time of the continuous semimartingale $r_N(X)$ at zero. It follows that if $\dimn \leq \dimm-2$ then the local time of $X$ on $N$ vanishes. In fact, if $X$ is nonexplosive with $\dimn=\dimm-1$ then, since $r_N(X)$ is a continuous semimartingale with $d[r_N (X)]_s=ds$, it follows that
\begin{equation}\label{eq:otamfd}
L^N_{t\wedge \tau} (X)=\lim_{\epsilon \downarrow 0} \frac{1}{2\epsilon} \int_0^{t\wedge \tau} \mathbf{1}_{B_\epsilon(N)} (X_s) \,ds
\end{equation}
for all $t\geq 0$, almost surely. Note that the Brownian motion considered here can be replaced with a Brownian motion with locally bounded and measurable drift.

\subsection{Revuz Measure}\label{ss:revuzmeas}

In this subsection we consider an application of a different approach, based on the theory of Markov processes. Although it is not always necessary to do so, we will assume that $M$ is compact. We will also assume that $N$ is a closed embedded hypersurface and that $X(x)$ is a Brownian motion on $M$ starting at $x$. The convexity based argument of \cite{MR1231929}, which was applied to the one point case, can be adapted to our situation and implies that with respect to the invariant measure $\vol_M$ the local time $L^N(X(x))$ corresponds in the sense of \cite{MR0279890} to the induced measure $\vol_N$. By \cite{MR1278079} this implies the following theorem, in which $p^M$ denotes the transition density function for Brownian motion.
\begin{theorem}\label{th:revuzimpthm}
Suppose that $M$ is compact, that $N$ is a closed embedded hypersurface and that $X(x)$ is a Brownian motion on $M$ starting at $x$. Then
\begin{equation}\label{eq:revuzimp}
\mathbb{E}\left[L^N_t(X(x))\right]= \int_0^t \int_N p_s^M(x,y)d\vol_N(y)\,ds
\end{equation}
for all $t\geq0$.
\end{theorem}
In a subsequent article we will calculate, estimate and provide an asymptotic relation for the rate of change $\frac{d}{dt}\mathbb{E}\left[L^N_t(X(x))\right]$. By a standard change of variables, it follows that the expected value of the occupation times appearing inside the limit on the right-hand side of \eqref{eq:otamfd} converge to the right-hand side of \eqref{eq:revuzimp} as $\epsilon \downarrow 0$. If one could justify exchanging this limit with the expectation, one would obtain a different proof of Theorem \ref{th:revuzimpthm}. The following corollary follows directly from Theorem \ref{th:revuzimpthm} and basic ergodicity properties of Brownian motion.
\begin{corollary}\label{cor:largetime}
Suppose that $M$ is compact, that $N$ is a closed embedded hypersurface and that $X$ is a Brownian motion on $M$. Then
\begin{equation*}
\lim_{t \uparrow \infty} \frac{1}{t}\mathbb{E}\left[L^N_t(X)\right] = \frac{\vol_N(N)}{\vol_M(M)}.
\end{equation*}
\end{corollary}
\begin{example}
Suppose $M=\mathbb{S}^1$ (i.e. the unit circle with the standard metric) and let $X(x)$ be a Brownian motion starting at $x \in \mathbb{S}^1$. By formula \eqref{eq:rNFormulaforBM} it follows that
\begin{equation*}
r^2_x(X_t(x)) = r_x^2(x) + 2\int_0^t r_x(X_s(x))d\beta_s + t-2\int_0^t r_x(X_s(x)) dL^{\Cut(x)}_s(X(x))
\end{equation*}
for $t \geq 0$, where $\beta$ is a standard one-dimensional Brownian motion. But $r_x(x)=0$ and $\Cut(x)$ is antipodal to $x$, which is a distance $\pi$ away from $x$, so as $dL^{\Cut(x)}(X(x))$ is supported on $\lbrace s\geq0 : X_s = \Cut(x)\rbrace$ we deduce that
\begin{equation}\label{eq:eqN1}
r^2_x(X_t(x)) = 2\int_0^t r_x(X_s(x))d\beta_s + t-2\pi L^{\Cut(x)}_t(X(x))
\end{equation}
for $t \geq 0$. Now $p^{\mathbb{S}^1}_t(x,\cdot) \rightarrow (2\pi)^{-1}$ as $t\uparrow \infty$ so
\begin{equation}\label{eq:eqN2}
\lim_{t \uparrow \infty} \mathbb{E}\left[r^2_x(X_t(x))\right] = \int_{\mathbb{S}^1}  \frac{r^2_x(y)}{2\pi} d \vol_{\mathbb{S}^1}(y) = \int_{-\pi}^{\pi}\frac{v^2}{2\pi}dv = \frac{\pi^2}{3}.
\end{equation}
Thus by equations \eqref{eq:eqN1} and \eqref{eq:eqN2} it follows that
\begin{equation*}
\frac{\pi^2}{3} = \lim_{t \uparrow \infty} \left(t-2\pi \mathbb{E}[L^{\Cut(x)}_t(X(x))]\right)
\end{equation*}
which implies for large times $t$ the approximation
\begin{equation*}
\mathbb{E} [L^{\Cut(x)}_t(X(x))] = \frac{t}{2 \pi}-\frac{\pi}{6}+o(1).
\end{equation*}
\end{example}
Corollary \ref{cor:largetime} and the previous example concern the behaviour of $\mathbb{E}\left[L^N(X)\right]$ for large times, in the compact case. In the next example we fix the time, instead considering the effect of expanding size of the submanifold, using balls in Euclidean space as the example.
\begin{example}
For $r>0$ denote by $\mathbb{S}^{\dimm-1}(r)$ the boundary of the open ball in $\mathbb{R}^\dimm$ of radius $r$ centred at the origin. If $X$ is a Brownian motion on $\mathbb{R}^\dimm$ starting at the origin then $r_{\mathbb{S}^{\dimm-1}(r)}(X)$ is a Markov process. It follows from general theory, as in \cite{MR1278079}, that formula (\ref{eq:revuzimp}) holds with $N=\mathbb{S}^{\dimm-1}(r)$ and $M=\mathbb{R}^\dimm$ and therefore
\begin{equation*}
\frac{1}{r}\mathbb{E}\left[L^{ \mathbb{S}^{\dimm-1}(r)}_t(X)\right] = \frac{\Gamma\left(\frac{\dimm}{2}-1,\frac{r^2}{2t}\right)}{\Gamma\left(\frac{\dimm}{2}\right)}
\end{equation*}
where $\Gamma(a,b) = \int_b^{\infty} s^{a-1}\,e^{-s}\,ds$ is the upper incomplete Gamma function. In this setting the process $L^{ \mathbb{S}^{\dimm-1}(r)}(X)$ corresponds to the local time of an Bessel process of dimension $\dimm$, started at the origin, at the value $r$. In particular, for the case $\dimm=2$ we obtain
\begin{equation}\label{eq:eulermaspre}
\frac{1}{r}\mathbb{E}\left[L^{ \mathbb{S}^{1}(r)}_t(X)\right] = \Gamma\left(0,\frac{r^2}{2t}\right).
\end{equation}
By differentiating the exponential of the right-hand side of equation \eqref{eq:eulermaspre} we can then deduce the curious relation
\begin{equation*}
\lim_{t\uparrow \infty} \left(\log\left(\frac{2t}{r^2}+1\right) - \frac{1}{r}\mathbb{E}\left[L^{ \mathbb{S}^{1}(r)}_t(X)\right]\right) = \gamma
\end{equation*}
where $\gamma$ denotes the Euler-Mascheroni constant.
\end{example}

\section{Probabilistic Estimates}\label{se:section3}

In this section we combine the geometric inequalities of the first section with the It\^{o}-Tanaka formula of the second section and deduce probabilistic estimates for the radial moments of Brownian motion with respect to a submanifold.

\subsection{A Log-Sobolev Inequality Approach}

It is fairly standard practice to deduce exponential integrability from a logarithmic-Sobolev inequality. In this subsection, we show how this can be done in a restricted version of the situation in which we are otherwise interested. For this, denote by $\lbrace P_t:t\geq 0\rbrace$ the heat semigroup on $M$ (acting on some suitable space of functions). 
\begin{theorem}\label{th:logsobest}
Suppose that $M$ is a complete and connected Riemannian manifold of dimension $\dimm$ and that $N$ is a closed embedded submanifold of $M$ of dimension $\dimn\in\lbrace 0,\ldots,\dimm-1\rbrace$. Assume that there exist constants $C_1,\Lambda \geq 0$ such that
\begin{equation}\label{eq:bakled}
\lbrace\Ric(\xi,\xi): \xi \in TM, \|\xi\|=1\rbrace \geq -(\dimm-1)C_1^2
\end{equation}
with at least one of the three conditions \textbf{(C1)}, \textbf{(C2)} or \textbf{(C3)} of Theorem \ref{th:estimatefinal} satisfied with $C_2=0$. Then
\begin{equation}\label{eq:expineq2}
P_t(e^{\theta r_N})(x)\leq \exp\left[\theta\left(r^2_N (x)+(\dimm-\dimn) t \right)^{\frac{1}{2}} + (\dimn \Lambda + (\dimm-1)C_1) \theta t/2 + \theta^2C(t)/2\right]
\end{equation}
for all $\theta,t\geq0$ and
\begin{equation}\label{eq:expineq3}
P_t(e^{\frac{\theta}{2} r^2_N})(x)\leq \exp\left[\frac{\theta \left(\left(r^2_N (x)+(\dimm-\dimn) t \right)^{\frac{1}{2}} + (\dimn \Lambda + (\dimm-1)C_1) \theta t/2\right)^2}{2(1-C(t)\theta)}\right]
\end{equation}
for all $0\leq \theta <C^{-1}(t)$, where
\begin{equation*}
C(t):=\frac{e^{(\dimm-1)C_1^2 t}-1}{(\dimm-1)C_1^2}.
\end{equation*}
\end{theorem}
\begin{proof}
Let $X(x)$ a Brownian motion starting at $x \in M$, let $\lbrace D_i\rbrace_{i=1}^\infty$ be an exhaustion of $M$ by regular domains and denote by $\tau_{D_i}$ the first exit time of $X(x)$ from $D_i$. Using It\^{o}'s formula and formula \eqref{eq:rNFormulaforBM}, Corollary \ref{cor:maincor} and Jensen's inequality, we see that
\begin{equation*}
\mathbb{E}\left[ r_N^2(X_{t\wedge \tau_{D_i}}(x) \right] \leq r_N^2 (x) + (\dimm-\dimn) t + (\dimn \Lambda + (\dimm-1)C_1) \int_{0}^{t} \mathbb{E}\left[ r_N^2(X_{s\wedge \tau_{D_i}}(x) \right]^{\frac{1}{2}} ds
\end{equation*}
for all $t\geq 0$. Since the term $r_N^2(x)+(\dimm-\dimn)t$ is increasing in $t$, Bihari's inequality \cite{MR0079154}, which is a nonlinear integral form of Gronwall's inequality, implies
\begin{equation*}
\mathbb{E}\left[ r_N^2(X_{t\wedge \tau_{D_i}}(x) \right] \leq \left(\left(r^2_N (x)+(\dimm-\dimn) t \right)^{\frac{1}{2}} + (\dimn \Lambda + (\dimm-1)C_1) t/2 \right)^2
\end{equation*}
for all $t\geq 0$, from which it follows that
\begin{equation}\label{eq:nonlin2}
P_t(r_N^2)(x) \leq \left(\left(r^2_N (x)+(\dimm-\dimn) t \right)^{\frac{1}{2}} + (\dimn \Lambda + (\dimm-1)C_1) t/2 \right)^2
\end{equation}
for all $t \geq 0$, by Fatou's lemma. Now, Bakry and Ledoux discovered (see \cite{MR2294794} or \cite{MR1453132}) that condition \eqref{eq:bakled} implies the heat kernel logarithmic Sobolev inequality
\begin{equation}\label{eq:LogSob}
\Ent_tf^2(x)\leq 2C(t)P_t\left(\| \nabla f\|^2\right)(x)
\end{equation}
for all $f \in C^\infty(M)$ and $t>0$. By a slight generalization of the classical argument of Herbst (see \cite{MR1767995}) it follows that for Lipschitz $F$ with $\| F\|_{Lip}\leq 1$ and $\theta \in \mathbb{R}$ we have
\begin{equation}\label{eq:smoothineq}
P_t(e^{\theta F})(x) \leq \exp \left[\theta P_tF(x)+\theta^2C(t)/2\right]
\end{equation}
for all $t \geq 0$ while it was proved in \cite{MR1305064} (see also \cite{MR1258492}) that the log-Sobolev inequality \eqref{eq:LogSob} implies
\begin{equation}\label{eq:expest}
P_t(e^{\frac{\theta}{2} F^2})(x) \leq \exp \left[\frac{\theta P_tF^2(x)}{2(1-C(t)\theta)}\right]
\end{equation}
for all $0\leq \theta <C^{-1}(t)$. Since $r_N$ is Lipschitz with $\| r_N \|_{Lip}=1$, inequality \eqref{eq:expineq3} follows from \eqref{eq:nonlin2} by the estimate \eqref{eq:expest} while inequality \eqref{eq:expineq2} is proved similarly, by applying Jensen's inequality to \eqref{eq:nonlin2} and using the estimate \eqref{eq:smoothineq}.
\end{proof}
To obtain exponential integrability for the heat kernel under relaxed curvature assumption we will use a different approach, which is developed in the next subsection. While the estimates \eqref{eq:expineq2} and \eqref{eq:expineq3} are, roughly speaking, the best we have under the conditions of Theorem \ref{th:logsobest}, the estimate \eqref{eq:expineq3} does not reduce to the correct expression in $\mathbb{R}^\dimm$, as we will see, and our later estimates will take into account positive curvature whereas the estimates of this subsection do not. Thus the later estimates are preferable from the point of view of geometric comparison. On the other hand, the estimates given by \cite[Theorem~8.62]{MR1715265}, which concern the case $N=\lbrace x \rbrace$, suggest that the `double exponential' feature of the estimates \eqref{eq:expineq2} and \eqref{eq:expineq3} (which is the inevitable result of using Herbst's argument and Bakry and Ledoux's log-Sobolev constant, as opposed to being a consequence of our moment estimates) is not actually necessary.

\subsection{First and Second Radial Moments}\label{ss:fsrm}

Suppose now that $X(x)$ is a Brownian motion on $M$ with locally bounded and measurable drift $b$ starting from $x \in M$, defined upto an explosion time $\zeta(x)$, and that $N$ is a closed embedded submanifold of $M$ of dimension $\dimn \in \lbrace 0,\ldots,\dimm-1 \rbrace$. We will assume for the majority of this section that there exist constants $\nu \geq 1$ and $\lambda \in \mathbb{R}$ such that the inequality
\begin{equation}\label{eq:triineq}
\left( \frac{1}{2}\triangle + b \right) r_N^2 \leq \nu + \lambda r_N^2
\end{equation}
holds off the cut locus. Unless otherwise stated, future references to the validity of this inequality will refer to it on the domain $M \setminus \Cut(N)$. If $b$ satisfies a linear growth condition in $r_N$ then geometric conditions under which such an inequality arises are given by Theorem \ref{th:estimatefinal} (see Corollary \ref{cor:maincor}), the content of which the reader might like to briefly review. In particular, there are various situations in which one can choose $\lambda=0$. Alternatively, if $N$ is a point and the Ricci curvature is bounded below by a constant $R$ then inequality \eqref{eq:triineq} holds with $\nu = \dimm$ and $\lambda=-R/3$, as stated by inequality \eqref{eq:simpleestimate1R}. Of course, if $N$ is an affine linear subspace of $\mathbb{R}^\dimm$ then inequality \eqref{eq:triineq} holds as an \textit{equality} with $\nu = \dimm-\dimn$ and $\lambda = 0$. Note that inequality \eqref{eq:triineq} does not imply nonexplosion of $X(x)$. This is clear by considering the products of stochastically incomplete manifolds with ones which are not. We therefore use localization arguments to deal with the possibility of explosion.
\begin{theorem}\label{th:secradmomthm}
Suppose there exists constants $\nu \geq 1$ and $\lambda \in \mathbb{R}$ such that inequality \eqref{eq:triineq} holds. Then
\begin{equation}\label{eq:secondmoment}
\mathbb{E} \left[ \mathbb{1}_{\lbrace t< \zeta(x) \rbrace}r_N^{2}(X_t(x))\right] \leq (r_N^2 (x) +\nu R(t))e^{\lambda t}
\end{equation}
for all $t\geq 0$, where $R(t):=(1-e^{-\lambda t})/\lambda$.
\end{theorem}
\begin{proof}
Let $\lbrace D_i \rbrace_{i=1}^{\infty}$ be an exhaustion of $M$ by regular domains and denote by $\tau_{D_i}$ the first exit time of $X(x)$ from $D_i$. Note that $\tau_{D_i} < \tau_{D_{i+1}}$ and that this sequence of stopping times announces the explosion time $\zeta(x)$. Then, by formula \eqref{eq:rNFormulaforBM}, it follows that
\begin{equation}\label{eq:itoforsecond}
\begin{split}
r_N^2 (X_{t\wedge \tau_{D_i}}(x)) =r_N^2 (x) &+ 2 \int_{0}^{t \wedge \tau_{D_i}} r_N (X_s(x)) \left( d\beta_s -d\mathbb{L}^{\Cut(N)}_s(X(x))\right)\\
&+ \int_{0}^{t \wedge \tau_{D_i}} \left( \frac{1}{2}\triangle + b \right) r_N^2 (X_s(x)) ds\\
\end{split}
\end{equation}
for all $t\geq 0$, almost surely. Since the domains $D_i$ are of compact closure the It\^{o} integral in \eqref{eq:itoforsecond} is a true martingale and it follows that
\begin{equation}\label{eq:note1}
\begin{split}
\mathbb{E}\left[r_N^2(X_{t\wedge \tau_{D_i}}(x))\right] = r_N^2(x) &- 2\,\mathbb{E}\left[ \int_0^{t\wedge \tau_{D_i}} r_N(X_s(x))d\mathbb{L}^{\Cut(N)}_s(X(x))\right] \\
&+  \int_{0}^{t} \mathbb{E}\left[\mathbb{1}_{\lbrace s < \tau_{D_i}\rbrace}\left( \frac{1}{2}\triangle + b \right) r_N^2 (X_s(x)) \right] ds
\end{split}
\end{equation}
for all $t\geq 0$, where exchanging the order of integrals in the last term is easily justified by the use of the stopping time and the assumptions of the theorem. Before applying Gronwall's inequality we should be careful, since we are allowing the coefficient $\lambda$ to be negative. For this, note that
\begin{equation}\label{eq:note2}
\mathbb{E}\left[r_N^2(X_{t\wedge \tau_{D_i}}(x))\right] = \mathbb{E}\left[\mathbb{1}_{\lbrace t < \tau_{D_i}\rbrace} r_N^2(X_t(x))\right]+\mathbb{E}\left[\mathbb{1}_{\lbrace t \geq \tau_{D_i}\rbrace} r_N^2(X_{\tau_{D_i}}(x))\right]
\end{equation}
and that the two functions
\begin{equation*}
t\mapsto \mathbb{E}\left[ \int_0^{t\wedge \tau_{D_i}} r_N(X_s(x))d\mathbb{L}^{\Cut(N)}_s(X(x))\right], \quad t\mapsto \mathbb{E}\left[\mathbb{1}_{\lbrace t \geq \tau_{D_i}\rbrace}r_N^2(X_{\tau_{D_i}}(x))\right]
\end{equation*}
are nondecreasing. If we define a function $f_{x,i,2}$ by
\begin{equation*}
f_{x,i,2}(t):= \mathbb{E}\left[ \mathbb{1}_{\lbrace t<\tau_{D_i}\rbrace} r_N^2(X_t(x))\right]
\end{equation*}
then $f_{x,i,2}$ is differentiable, since the boundaries of the $D_i$ are smooth, and it follows from \eqref{eq:note1} and \eqref{eq:note2} that we have the differential inequality
\begin{equation}\label{eq:diffineq2}
\begin{cases}
f'_{x,i,2}(t)\leq \nu + \lambda f_{x,i,2}(t)\\
f_{x,i,2}(0)=r_N^2(x)
\end{cases}
\end{equation}
for all $t\geq 0$. Now applying Gronwall's inequality to \eqref{eq:diffineq2} yields
\begin{equation}\label{eq:basecase}
\mathbb{E} \left[ \mathbb{1}_{\lbrace t<\tau_{D_i}\rbrace}r_N^2(X_t(x)) \right] \leq r_N^2 (x)e^{\lambda t} +\nu\left(\frac{e^{\lambda t}-1}{\lambda}\right)
\end{equation}
for all $t\geq0$, from which the result follows by the monotone convergence theorem.
\end{proof}
\begin{remark}
If one wishes to include on the right-hand of inequality \eqref{eq:triineq} a term that is linear in $r_N$, as in the estimate \eqref{eq:simpleestimate1}, or simply a continuous function of $r_N$, as in the estimates \eqref{eq:estinkappa1}, \eqref{eq:estinkappa2} and \eqref{eq:estinkappa3}, with suitable integrability properties, then one can do so and use a nonlinear version of Gronwall's inequality, such as Bihari's inequality, to obtain an estimate on the left-hand side of \eqref{eq:secondmoment}.
\end{remark}
We will refer the object on the left-hand side of inequality \eqref{eq:secondmoment} as \textit{the second radial moment of $X(x)$ with respect to $N$}. To find an inequality for \textit{the first radial moment of $X(x)$ with respect to $N$} one can simply use Jensen's inequality. Note that $\lim_{\lambda \rightarrow 0} R(t)e^{\lambda t}= t$ and this provides the sense in which Theorem \ref{th:secradmomthm} and similar statements should be interpreted if $\lambda =0$.

\subsection{Nonexplosion}\label{ss:exittimes}

That a Ricci curvature lower bound implies stochastic completeness was originally proved by Yau in \cite{MR505904}, as mentioned earlier. This was extended by Ichihara in \cite{MR676590} and Hsu in \cite{MR1009455} to allow the Ricci curvature to grow in the negative direction in a certain way (like, for example, a negative quadratic in the distance function). Thus the following theorem is well-known in the one point case (in terms of which it can be proved). Using Theorem \ref{th:secradmomthm} our proof is short, so we may as well include it so as to keep the presentation of this article reasonably self-contained. So suppose that $r>0$, let $B_r(N):=\lbrace y \in M : r_N(y) < r\rbrace$ and denote by $\tau_{B_r(N)}$ the first exit time of $X(x)$ from the tubular neighbourhood $B_r(N)$.
\begin{theorem}\label{th:nonexp}
Suppose that $N$ is compact and that there exist constants $\nu\geq 1 $ and $\lambda \in \mathbb{R}$ such that inequality \eqref{eq:triineq} holds. Then $X(x)$ is nonexplosive.
\end{theorem}
\begin{proof}
By following the proof of Theorem \ref{th:secradmomthm}, with the stopping times $\tau_{D_i}$ replaced by $\tau_{B_i(N)}$, one deduces
\begin{equation*}
\mathbb{P} \lbrace \tau_{B_i(N)} \leq t \rbrace  \leq \frac{\left(r_N^2(x)+\nu R(t)\right) e^{\lambda t}}{i^2}
\end{equation*}
for all $t\geq0$. This crude exit time estimate implies that $X(x)$ is nonexplosive since the compactness of $N$ implies that the stopping times $\tau_{B_i(N)}$ announce the explosion time $\zeta(x)$.
\end{proof}

\subsection{Higher Radial Moments}\label{ss:hrm}

Recall that if $X$ is a real-valued Gaussian random variable with mean $\mu$ and variance $\sigma^2$ then for $\mom \in \mathbb{N}$ one has the formula
\begin{equation}\label{eq:absolutelag}
\mathbb{E}\left[  X ^{2\mom} \right] = \left(2\sigma^2\right)^\mom \mom! L^{-\frac{1}{2}}_\mom \left(-\frac{\mu^2}{2\sigma^2}\right)
\end{equation}
where $L^{\alpha}_{\mom}(z)$ are the \textit{Laguerre polynomials}, defined by the formula
\begin{equation*}
L^{\alpha}_{\mom}(z) = e^z \frac{z^{-\alpha}}{\mom!}\frac{\partial^\mom}{\partial z^\mom}\left(e^{-z}z^{\mom+\alpha}\right)
\end{equation*}
for $\mom=0,1,2,\ldots$ and $\alpha >-1$ (for the properties of Laguerre polynomials used in this article, see \cite{MR0350075}). In particular, if $X(x)$ is a standard Brownian motion on $\mathbb{R}$ starting from $x \in \mathbb{R}$ then
\begin{equation*}
\mathbb{E} \left[\vert X_t(x) \vert^{2\mom} \right]= \left(2t\right)^\mom \mom! L^{-\frac{1}{2}}_\mom \left(-\frac{\vert x\vert^2}{2t}\right)
\end{equation*}
for all $t\geq0$. With this in mind we prove the following theorem, a special case of which is Theorem \ref{th:secradmomthm}, which will be used in the next section to obtain exponential estimates. Theorem \ref{th:secradmomthm} was stated separately because it constitutes the base case in an induction argument.
\begin{theorem}\label{th:evenradmomthm}
Suppose that there exist constants $\nu \geq1$ and $\lambda \in \mathbb{R}$ such that inequality \eqref{eq:triineq} holds and let $\mom \in \mathbb{N}$. Then
\begin{equation}\label{eq:evenmom}
\mathbb{E}\left[ \mathbb{1}_{\lbrace t<\zeta(x)\rbrace}r_N^{2\mom}(X_t(x))\right] \leq  \left(2R(t)e^{\lambda t}\right)^\mom \mom! L^{\frac{\nu}{2}-1}_\mom \left(-\frac{r_N^2(x)}{2R(t)}\right)
\end{equation}
for all $t\geq 0$, where $R(t):=(1-e^{-\lambda t})/\lambda$.
\end{theorem}
\begin{proof}
By the assumption of the theorem we see that on $M \setminus \Cut(N)$ and for $\mom \in \mathbb{N}$ we have
\begin{equation*}
\left( \frac{1}{2}\triangle + b\right) r_N^{2\mom} \leq \mom\left(\nu + 2\left(\mom-1\right)\right)r_N^{2\mom-2} + \mom \lambda r_N^{2\mom},
\end{equation*}
and by formula \eqref{eq:rNFormulaforBM} we have
\begin{equation*}
\begin{split}
r_N^{2\mom} (X_{t\wedge \tau_{D_i}}(x)) = r_N^{2\mom} (x) &+ 2\mom \int_{0}^{t \wedge \tau_{D_i}} r_N^{2\mom-1} (X_s(x)) \left(d\beta_s- d\mathbb{L}_{s}^{\Cut(N)}(X(x))\right) \\
& + \int_{0}^{t \wedge \tau_{D_i}} \left( \frac{1}{2}\triangle + b \right) r_N^{2\mom} (X_s(x)) ds
\end{split}
\end{equation*}
for all $t\geq 0$, almost surely, where the stopping times $\tau_{D_i}$ are defined as in the proof of Theorem \ref{th:secradmomthm}. It follows that if we define functions $f_{x,i,2\mom}$ by
\begin{equation*}
f_{x,i,2\mom}(t):= \mathbb{E}\left[\mathbb{1}_{\lbrace t<\tau_{D_i}\rbrace} r_N^{2\mom}(X_t(x))\right]
\end{equation*}
then, arguing as we did in the proof of Theorem \ref{th:secradmomthm}, there is the differential inequality
\begin{equation*}
\begin{cases}
f'_{x,i,2\mom}(t)\leq \mom\left(\nu + 2\left(\mom-1\right)\right) f_{x,i,2(\mom-1)}(t) + \mom\lambda f_{x,i,2\mom}(t)\\
f_{x,i,2\mom}(0)=r_N^{2\mom}(x)
\end{cases}
\end{equation*}
for all $t\geq 0$. Applying Gronwall's inequality yields
\begin{equation}\label{eq:evenmoments}
f_{x,i,2\mom}(t) \leq \left(r_N^{2\mom} (x) +  \mom\left(\nu + 2\left(\mom-1\right)\right) \int_0^t f_{x,i,2(\mom-1)}(s) e^{-\mom\lambda s} ds\right)e^{\mom\lambda t}
\end{equation}
for all $t \geq 0$ and $\mom \in \mathbb{N}$. The next step in the proof is to use induction on $\mom$ to show that
\begin{equation}\label{eq:inductionineq}
f_{x,i,2\mom}(t) \leq \sum_{\sv=0}^{\mom} \binom{\mom}{\sv} \left(2R(t)\right)^{\mom-\sv} r^{2\sv}_N (x) \frac{\Gamma (\frac{\nu}{2}+\mom)}{\Gamma(\frac{\nu}{2}+\sv)}e^{\mom\lambda t} 
\end{equation}
for all $t \geq 0$ and $\mom \in \mathbb{N}$. Inequality \eqref{eq:basecase} covers the base case $\mom=1$. If we hypothesise that the inequality holds for some $\mom-1$ then by inequality \eqref{eq:evenmoments} we have
\begin{equation}\label{eq:prelimind}
f_{x,i,2\mom}(t) \leq \bigg(r_N^{2\mom} (x) + \mom\left(\nu + 2\left(\mom-1\right)\right)\sum_{\sv=0}^{\mom-1} \binom{\mom-1}{\sv} r^{2\sv}_N (x) \frac{\Gamma (\frac{\nu}{2}+\mom-1)}{\Gamma(\frac{\nu}{2}+\sv)}\tilde{R}(t) \bigg)e^{\mom\lambda t}
\end{equation}
for all $t\geq 0$, where $\tilde{R}(t)= \int_0^t\left(2R(s)\right)^{\mom-1-\sv}e^{-\lambda s} ds$. Using $2(\mom-\sv)\tilde{R}(t) = (2R(t))^{\mom-\sv}$ and properties of the Gamma function it is straightforward to deduce inequality \eqref{eq:inductionineq} from inequality \eqref{eq:prelimind} which completes the inductive argument. Since $\nu \geq 1$ we can then apply the relation 
\begin{equation*}
L^{\alpha}_\mom (z) = \sum_{\sv=0}^{\mom} \frac{\Gamma(\mom+\alpha+1)}{\Gamma(\sv+\alpha+1)}\frac{(-z)^\sv}{\sv!(\mom-\sv)!},
\end{equation*}
which can be proved using Leibniz's formula, to see that
\begin{equation*}
\sum_{\sv=0}^{\mom} \binom{\mom}{\sv} (2R(t))^{\mom-\sv} r_N^{2\sv}(x) \frac{\Gamma (\frac{\nu}{2}+\mom)}{\Gamma(\frac{\nu}{2}+\sv)} =  (2R(t))^\mom \mom! L^{\frac{\nu}{2}-1}_\mom \left(-\frac{r_N^2(x)}{2R(t)}\right)
\end{equation*}
and so by inequality \eqref{eq:inductionineq} it follows that
\begin{equation}\label{eq:afhyp}
f_{x,i,2\mom}(t) \leq  \left(2R(t)e^{\lambda t}\right)^\mom \mom! L^{\frac{\nu}{2}-1}_\mom \left(-\frac{r_N^2(x)}{2R(t)}\right)
\end{equation}
for $t\geq0$ and $i,\mom \in \mathbb{N}$. The result follows from this by the monotone convergence theorem.
\end{proof}
We will refer the object on the left-hand side of inequality \eqref{eq:evenmom} as \textit{the $2\mom$-th radial moment of $X(x)$ with respect to $N$}. One can deduce an estimate for \textit{the $(2\mom-1)$-th radial moment of $X(x)$ with respect to $N$} by Jensen's inequality. We conclude this subsection with an example where the even radial moments can be calculated explicitly.
\begin{example}\label{ex:hyp3ex}
Denote by $\mathbb{H}^{3}_\kappa$ the $3$-dimensional hyperbolic space with constant sectional curvatures equal to $\kappa <0$ and suppose that $X(x)$ is a Brownian motion on $\mathbb{H}^{3}_\kappa$ starting at $x$. Then the densities of $X(x)$ are given by the deterministic heat kernel formula
\begin{equation}\label{eq:hkonhyp3}
p^{\mathbb{H}^{3}_\secur}_t(x,y) = \Theta^{-\frac{1}{2}}_y(x)(2\pi t)^{-\frac{3}{2}}\exp\left[-\frac{d^2(x,y)}{2t}+\frac{\secur t}{2}\right]
\end{equation}
for $x,y \in \mathbb{H}^{3}_\secur$ and $t>0$ where
\begin{equation}\label{eq:thetaonhyp3}
\Theta_y(x) = \left(\frac{\sinh \left(\sqrt{-\secur}r_y(x)\right)}{\sqrt{-\secur}r_y(x)}\right)^2.
\end{equation}
By a change of variables it follows that for each $\mom \in \mathbb{N}$ we have
\begin{equation*}
\mathbb{E}[r^{2\mom}_x (X_t(x))] = (2t)^\mom \frac{\Gamma\left(\frac{3}{2}+\mom\right)}{\Gamma\left(\frac{3}{2}\right)}\mathstrut_1 F_1 \left(\frac{3}{2}+\mom,\frac{3}{2},-\frac{\kappa t}{2}\right)
\end{equation*}
for all $t\geq0$, where $\mathstrut_1 F_1$ is the confluent hypergeometric function of the first kind, so in particular
\begin{equation*}
\mathbb{E}[r_x^2(X_t(x))] = 3t-\kappa t^2
\end{equation*}
for all $t\geq0$. This ties in with the fact, proved by Liao and Zheng in \cite{MR1330766}, that on general $M$ if $X(x)$ is a Brownian motion starting at $x \in M$ and if $\tau_\epsilon$ is the first exit time of $X(x)$ from the geodesic ball $B_\epsilon(x)$ then
\begin{equation*}
\mathbb{E}\left[r_x^2(X_{\tau_{\epsilon}\wedge t}(x))\right]= \dimm t-\frac{1}{6}\scal(x)t^2+o(t^2)
\end{equation*}
as $t\downarrow 0$, where $o(t^2)$ might depend upon $\epsilon$ and where $\scal(x)$ denotes the scalar curvature at $x$ (since on $\mathbb{H}^{3}_\kappa$ the scalar curvature is constant and equal to $6\kappa$).
\end{example}

\subsection{Exponential Estimates}\label{ss:exprm}

Before using the estimates of the previous subsection to obtain exponential inequalities, we need the following lemma.
\begin{lemma}\label{lem:techlemma}
For $\alpha,z\geq 0$ and $m=1,2,\ldots$ we have
\begin{equation*}
\mom! L^{\alpha}_\mom(-z) \leq \left(12\left(1+z\right)\right)^\mom\frac{\Gamma\left(\alpha+1+\mom\right)}{\Gamma\left(\alpha +1\right)}.
\end{equation*}
\end{lemma}
\begin{proof}
Recall that
\begin{equation}\label{eq:expres}
\mom!L^{\alpha}_\mom(-z) = \sum_{\sv=0}^\mom \binom{\mom}{\sv}z^\sv \frac{\Gamma\left(\alpha+1+\mom\right)}{\Gamma\left(\alpha +1 + \sv\right)}.
\end{equation}
Now, since $\alpha,z \geq 0$ it follows that $\Gamma\left(\alpha + 1+\sv\right) \geq \Gamma\left(\alpha+1\right)$ and that $z^\sv\leq (1+z)^\mom$ for all $\sv \in \lbrace 0,\ldots,\mom\rbrace$. For such $\sv$ there is the bound
\begin{equation*}
\binom{\mom}{\sv}\leq \left(\frac{\mom e}{\sv}\right)^\sv 
\end{equation*}
and since the largest binomial coefficient is `the middle one' it follows that
\begin{equation*}
\binom{\mom}{\sv}\leq
\begin{cases}
\left(\frac{2\mom e}{\mom+1}\right)^{\frac{\mom+1}{2}}& \text{ if $\mom$ is odd}\\
\left(2e\right)^{\frac{\mom}{2}} &\text{ if $\mom$ is even}\\
\end{cases}
\end{equation*}
and therefore we have the somewhat crude bound
\begin{equation*}
\binom{\mom}{\sv}\leq 6^\mom 
\end{equation*}
for $\sv\in \lbrace 0,\ldots,\mom\rbrace$. Substituting these bounds into equation \eqref{eq:expres} yields
\begin{equation*}
\mom!L^{\alpha}_\mom(-z) \leq (\mom+1) (6(1+z))^\mom\frac{\Gamma\left(\alpha+1+\mom\right)}{\Gamma\left(\alpha +1\right)}
\end{equation*}
from which the lemma follows since $\mom+1\leq 2^\mom$.
\end{proof}
\begin{theorem}\label{th:inequality2}
Suppose that there exist constants $\nu \geq2$ and $\lambda \in \mathbb{R}$ such that inequality \eqref{eq:triineq} holds. Then
\begin{equation}\label{eq:estfordistexp}
\mathbb{E}\left[\mathbb{1}_{\{t< \zeta(x) \}}e^{\theta r_N(X_t(x))}\right]\leq 1 +\left( 1+\mathbf{R}(t,\theta,x)^{-\frac{1}{2}}\right)\left(\mathstrut_1 F_1\left(\frac{\nu}{2},\frac{1}{2}, \mathbf{R}(t,\theta,x)\right)-1\right)
\end{equation}
for all $t,\theta\geq 0$, where
\begin{equation}\label{eq:boldRdefn}
\mathbf{R}(t,\theta,x):=12\theta^2 \left(r_N^2(x)+2R(t)\right)e^{\lambda t}
\end{equation}
with $R(t):=(1-e^{-\lambda t})/\lambda$ and where $\mathstrut_1 F_1$ is the confluent hypergeometric function of the first kind.
\end{theorem}
\begin{proof}
With the stopping times $\tau_{D_i}$ as in the proof of Theorem \ref{th:secradmomthm}, for $\mom \in \mathbb{N}$ with $\mom$ even we see by inequality \eqref{eq:afhyp} that
\begin{equation*}
\mathbb{E}\left[ \mathbf{1}_{\lbrace t<\tau_{D_i}\rbrace}r_x^{\mom}(X_t(x))\right] \leq \left(2R(t)e^{\lambda t}\right)^{\frac{\mom}{2}} \Gamma\left(\frac{\mom}{2}+1\right) L^{\frac{\nu}{2}-1}_{\frac{\mom}{2}} \left(-\frac{r_N^2(x)}{2R(t)}\right)
\end{equation*}
and so, by Jensen's inequality, if $\mom$ is odd then
\begin{equation*}
\mathbb{E}\left[ \mathbf{1}_{\lbrace t<\tau_{D_i}\rbrace}r_x^{\mom}(X_t(x))\right] \leq \left(2R(t)e^{\lambda t}\right)^{\frac{\mom}{2}} \left(\Gamma\left(\frac{\mom+1}{2}+1\right) L^{\frac{\nu}{2}-1}_{\frac{\mom+1}{2}} \left(-\frac{r_N^2(x)}{2R(t)}\right)\right)^{\frac{\mom}{\mom+1}}.
\end{equation*}
It follows from this and Lemma \ref{lem:techlemma}, since $\nu \geq 2$, that
\begin{equation*}
\begin{split}
&\quad\quad\quad \text{ } \mathbb{E}\left[\mathbf{1}_{\{t< \tau_{D_i} \}}e^{\theta r_N(X_t(x))}\right]\\
&\leq 1 + \sum_{\mom=1}^{\infty} \frac{\left(2\theta^2R(t)e^{\lambda t}\right)^\mom}{(2\mom)!}\mom!L^{\frac{\nu}{2}-1}_\mom\left(-\frac{r_N^2(x)}{2R(t)}\right)\\
&\quad\quad+ \sum_{\mom=1}^{\infty} \frac{\left(2\theta^2R(t)e^{\lambda t}\right)^{\frac{2\mom-1}{2}}}{(2\mom-1)!}\left(\mom!L^{\frac{\nu}{2}-1}_\mom\left(-\frac{r_N^2(x)}{2R(t)}\right)\right)^{\frac{2\mom-1}{2\mom}}\\
&\leq 1 + \sum_{\mom=1}^{\infty} \frac{\left(2\theta^2R(t)e^{\lambda t}\right)^\mom}{(2\mom)!} \left(12\left(1+\frac{r_N^2\left(x\right)}{2R(t)}\right)\right)^\mom\frac{\Gamma\left(\frac{\nu}{2}+\mom\right)}{\Gamma\left(\frac{\nu}{2}\right)}\\
&\quad\quad+ \sum_{\mom=1}^{\infty} \frac{\left(2\theta^2R(t)e^{\lambda t}\right)^{\frac{2\mom-1}{2}}}{(2\mom-1)!}\left( \left(12\left(1+\frac{r_N^2\left(x\right)}{2R(t)}\right)\right)^\mom\frac{\Gamma\left(\frac{\nu}{2}+\mom\right)}{\Gamma\left(\frac{\nu}{2}\right)}\right)^{\frac{2\mom-1}{2\mom}}\\
&=\quad\quad \sum_{\mom=0}^{\infty} \frac{\left(24\theta^2\left(R(t)+\frac{r_N^2\left(x\right)}{2}\right)e^{\lambda t}\right)^\mom}{(2\mom)!} \frac{\Gamma\left(\frac{\nu}{2}+\mom\right)}{\Gamma\left(\frac{\nu}{2}\right)}\\
&\quad\quad+ \sum_{\mom=1}^{\infty} \frac{\left(24\theta^2\left(R(t)+\frac{r_N^2\left(x\right)}{2}\right)e^{\lambda t}\right)^{\frac{2\mom-1}{2}}}{(2\mom-1)!}\left(\frac{\Gamma\left(\frac{\nu}{2}+\mom\right)}{\Gamma\left(\frac{\nu}{2}\right)}\right)^{\frac{2\mom-1}{2\mom}}.
\end{split}
\end{equation*}
Now using $(2\mom)!= 2\mom(2\mom-1)!$, $2\mom \leq 4^\mom$ and $\Gamma\left(\frac{\nu}{2}+\mom\right)\geq \Gamma\left(\frac{\nu}{2}\right)$, together with the relation
\begin{equation*}
\sum_{\mom=0}^\infty \frac{z^\mom}{(2\mom)!}\frac{\Gamma\left(\frac{\nu}{2}+\mom\right)}{\Gamma\left(\frac{\nu}{2}\right)} = \mathstrut_1 F_1 \left(\frac{\nu}{2}, \frac{1}{2},\frac{z}{4}\right),
\end{equation*}
which can be seen directly from the definition of $\mathstrut_1 F_1$ as a generalized hypergeometric series, we deduce that
\begin{equation}\label{eq:rhsasasm}
\begin{split}
\mathbb{E}\left[\mathbf{1}_{\{t< \tau_{D_i} \}}e^{\theta r_N(X_t(x))}\right]&\leq \sum_{\mom=0}^\infty \frac{(4\mathbf{R}(t,\theta,x))^p}{(2\mom)!}\frac{\Gamma\left(\frac{\nu}{2}+\mom\right)}{\Gamma \left(\frac{\nu}{2}\right)}\\
&\quad+ \mathbf{R}(t,\theta,x)^{-\frac{1}{2}}\sum_{\mom=1}^\infty \frac{(4\mathbf{R}(t,\theta,x))^p}{(2\mom)!}\frac{\Gamma\left(\frac{\nu}{2}+\mom\right)}{\Gamma \left(\frac{\nu}{2}\right)}\\
&= 1 +\left( 1+\mathbf{R}(t,\theta,x)^{-\frac{1}{2}}\right)\left(\mathstrut_1 F_1\left(\frac{\nu}{2},\frac{1}{2}, \mathbf{R}(t,\theta,x)\right)-1\right)
\end{split}
\end{equation}
and the theorem follows by monotone convergence. \qed
\end{proof}
\begin{remark}
The right-hand side of \eqref{eq:estfordistexp} is a continuous function of $t$, $\theta$ and $x$ and since the function $\mathstrut_1 F_1$ satisfies $\mathstrut_1 F_1\left(\frac{\nu}{2},\frac{1}{2},0\right) = 1$ and
\begin{equation*}
\lim_{r \downarrow 0}r^{-\frac{1}{2}}\left(\mathstrut_1 F_1\left(\frac{\nu}{2},\frac{1}{2},r\right)-1\right)=0
\end{equation*}
it follows that if $x\in N$ then the right-hand side of inequality \eqref{eq:estfordistexp} converges to the limit of the left-hand side as $t \downarrow 0$. Furthermore, for the values of $\nu$ considered in the theorem the right-hand side of \eqref{eq:estfordistexp} grows exponentially with $\mathbf{R}(t,\theta,x)$ (in particular $\mathstrut_1 F_1(1/2,1/2,z)=e^z$).The theorem shows that under the assumptions of the theorem there is no positive time at which the left-hand side of \eqref{eq:estfordistexp} is infinite.
\end{remark}
For $\vert \gamma \vert<1$ the Laguerre polynomials also satisfy the identity
\begin{equation}\label{eq:lagproptwo}
\sum_{\mom=0}^{\infty} \gamma^\mom L^{\alpha}_\mom(z) = (1-\gamma)^{-(\alpha+1)} e^{-\frac{z\gamma}{1-\gamma}},
\end{equation}
which is proved in \cite{MR0350075}. It follows from this identity and equation \eqref{eq:absolutelag} that for a real-valued Gaussian random variable $X$ with mean $\mu$ and variance $\sigma^2$ we have for $\theta \geq 0$ that
\begin{equation*}
\mathbb{E} \left[e^{\frac{\theta}{2} \vert X\vert^2}\right] = \left(1-\theta \sigma^2\right)^{-\frac{1}{2}}\exp\left[\frac{\theta\vert \mu \vert ^2 }{2(1-\theta \sigma^2)}\right]
\end{equation*}
so long as $\theta \sigma^2 <1$ (and there is a well-known generalization of this formula for Gaussian measures on Hilbert spaces). In particular, if $X(x)$ is a standard Brownian motion on $\mathbb{R}$ starting from $x \in \mathbb{R}$ then for $t\geq0$ it follows that
\begin{equation*}
\mathbb{E} \left[e^{\frac{\theta}{2} \vert X_t(x)\vert^2}\right] = \left(1-\theta t\right)^{-\frac{1}{2}}\exp\left[\frac{\theta \vert x \vert^2 }{2(1-\theta t)}\right]
\end{equation*}
so long as $\theta t <1$. With this in mind we state the following theorem.
\begin{theorem}\label{th:expenbm}
Suppose there exists constants $\nu \geq1$ and $\lambda \in \mathbb{R}$ such that inequality \eqref{eq:triineq} holds. Then
\begin{equation}\label{eq:exponential2}
\mathbb{E} \left[\mathbb{1}_{\{t< \zeta(x) \}} e^{\frac{\theta}{2} r_N^2 (X_t(x))}\right] \leq \left(1-\theta R(t) e^{\lambda t}\right)^{-\frac{\nu}{2}}\exp\left[\frac{\theta r_N^2(x)e^{\lambda t}}{2(1-\theta R(t) e^{\lambda t})}\right]
\end{equation}
for all $t,\theta \geq 0$ such that $\theta R(t) e^{\lambda t}<1$, where $R(t):=(1-e^{-\lambda t})/\lambda$.
\end{theorem}
\begin{proof}
Using inequality \eqref{eq:afhyp} and equation \eqref{eq:lagproptwo} we see that
\begin{equation}
\begin{split}
\text{ }\mathbb{E} \left[\mathbb{1}_{\{t< \tau_{D_i} \}} e^{\frac{\theta}{2} r_N^2 (X_t(x))}\right] =&\text{ }\sum_{\mom=0}^{\infty} \frac{\theta^\mom}{2^\mom \mom!} f_{x,i,2\mom}(t) \\
\leq&\text{ }\sum_{\mom=0}^{\infty} \left(\theta R(t)e^{\lambda t}\right)^\mom L^{\frac{\nu}{2}-1}_\mom \left(-\frac{r_N^2(x)}{2R(t)}\right) \\
=&\text{ } \left(1-\theta R(t) e^{\lambda t}\right)^{-\frac{\nu}{2}}\exp\left[\frac{\theta r_N^2(x)e^{\lambda t}}{2(1-\theta R(t) e^{\lambda t})}\right] \nonumber
\end{split}
\end{equation}
where we safely switched the order of integration using the stopping time. The result follows by the monotone convergence theorem.
\end{proof}
Theorem \ref{th:expenbm} improves upon the estimate given by the second part of a similar theorem due to Stroock (see \cite[Theorem~5.40]{MR1715265}), since the latter concerns only the one point case, does not take into account positive curvature or the possibility of drift and does not reduce to the correct expression in flat Euclidean space.
\begin{example}
For a qualitative picture of the behaviour of the above estimates, fix $x \in M$ and consider the case in which $X(x)$ is a Brownian motion starting at $x$. Denote by $R$ the infimum of the Ricci curvature of $M$ and assume that $R>-\infty$. Then inequality \eqref{eq:simpleestimate1R} implies that the assumptions of Theorems \ref{th:inequality2} and \ref{th:expenbm} hold when $N=\lbrace x\rbrace$ with $\nu = \dimm$ and $\lambda = -R/3$. For these parameters we plot in Figure \ref{fig:figh} the right-hand sides of the inequalities \eqref{eq:estfordistexp} and \eqref{eq:exponential2} as functions of time for the three cases $R\in \lbrace -1,0,1\rbrace$ with $\theta = \frac{1}{6}$ and $\dimm=3$. Note that if $R>0$ then the left-hand sides of these inequalities are bounded, by Myer's theorem.
\begin{figure}[ht]
\begin{center}
\makebox[\textwidth]{
\includegraphics[width=0.49\textwidth]{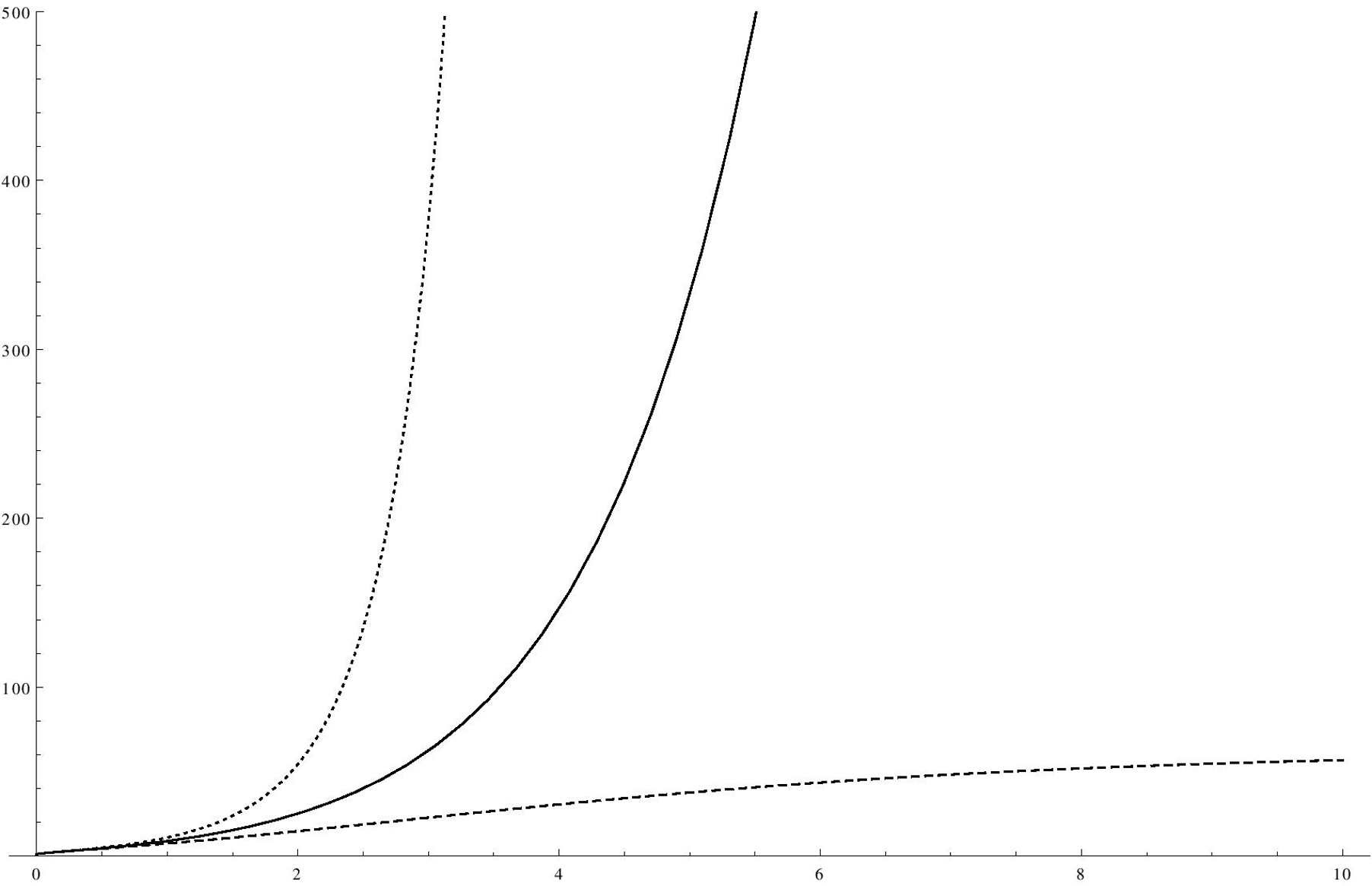}
\hfill
\includegraphics[width=0.49\textwidth]{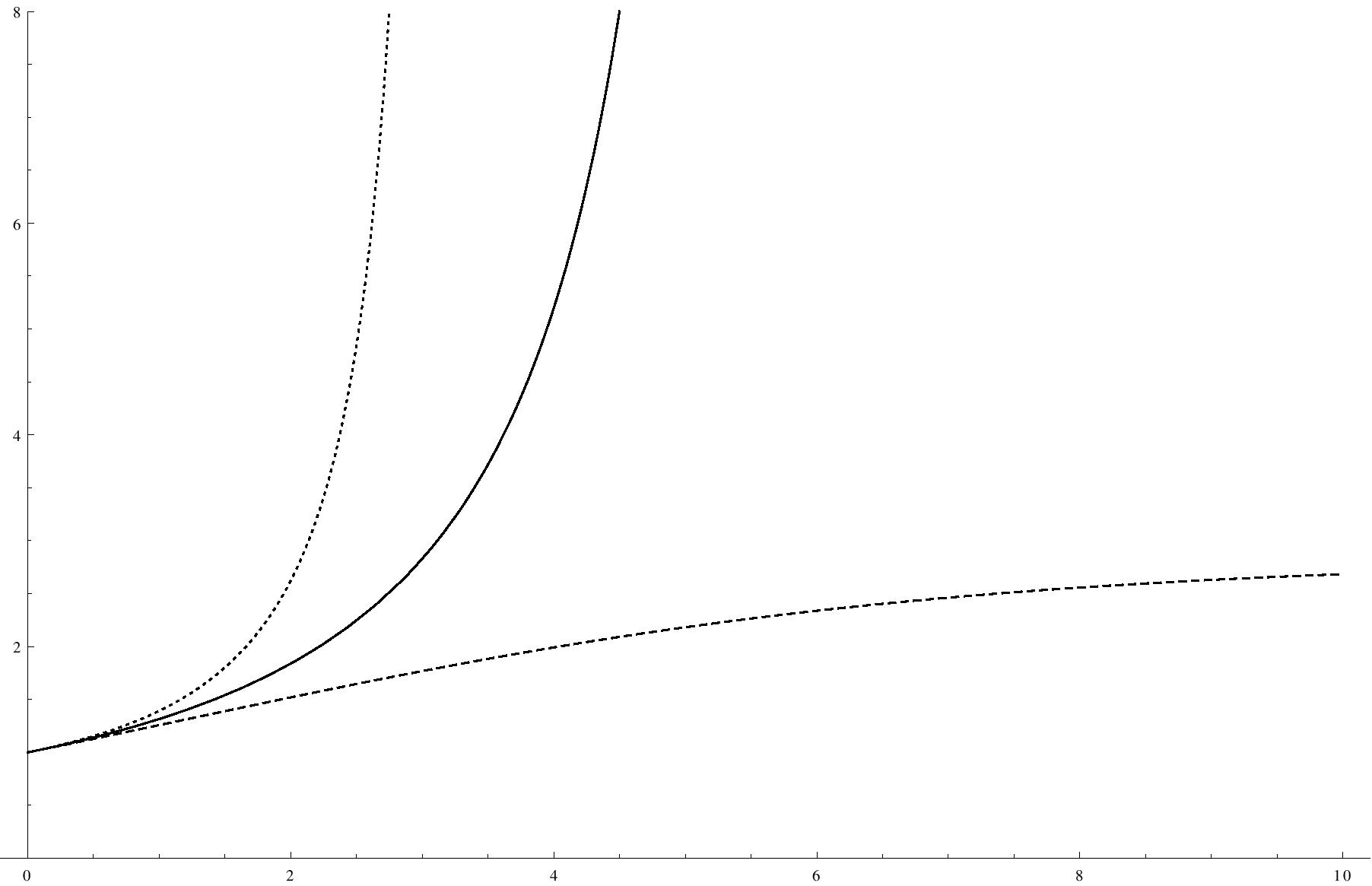}
}
\caption{\small Consider the case $N=\lbrace x\rbrace$ with $\nu = \dimm$ and $\lambda = -R/3$. The solid curve on the left represents the graph of the right-hand side of inequality \eqref{eq:estfordistexp} with $R=0$. Above it is a dotted curve, which is the analogous object for $R=-1$, and below it is a dashed curve, which is the analogous object for $R=1$. The solid curve on the right represents the graph of the right-hand side of inequality \eqref{eq:exponential2} with $R=0$. Above it is a dotted curve, which is the analagous object for $R=-1$, and below it is a dashed curve, which is the analagous object for $R=1$. We have set $\theta = \frac{1}{6}$ and $\dimm=3$ in all cases and the horizontal axes represent the time $t$. Although not obvious from the two plots, the dotted and solid curves plotted on the left do not explode in finite time while the dotted and solid curves plotted on the right explode at times $t=3\log 3\simeq 3.3$ and $t=6$ respectively.}
\label{fig:figh}
\end{center}
\end{figure}
\normalsize
\end{example}
The sharpness of our estimates allows for the following comparison inequality, for which we note that if $\dimn=0$ then $N$ is vacuously both totally geodesic and minimal, since in this case it would be an at most countable collection of isolated points.
\begin{corollary}
Suppose that $X(x)$ is a Brownian motion on $M$ starting at $x$ and that one of the following conditions is satisfied:
\begin{description}
\item[(I)] $\dimn \in \lbrace 0,\ldots,\dimm-1\rbrace$, the sectional curvature of planes containing the radial direction is non-negative and $N$ is totally geodesic;
\item[(II)] $\dimn \in \lbrace 0,\dimm-1\rbrace$, the Ricci curvature in the radial direction is non-negative and $N$ is minimal.
\end{description}
If $Y(y)$ denotes a Brownian motion on $\mathbb{R}^{\dimm-\dimn}$ starting at $y \in \mathbb{R}^{\dimm-\dimn}$ with $r_N^2(x) \leq \| y\|_{\mathbb{R}^{\dimm-\dimn}}$ then
\begin{equation*}
\mathbb{E} \left[e^{\frac{\theta}{2} r_N^2 (X_t(x))}\right] \leq \mathbb{E} \left[e^{\frac{\theta}{2} \|Y_t(y)\|^2_{\mathbb{R}^{\dimm-\dimn}}}\right]
\end{equation*}
for all $t,\theta \geq 0$ such that $\theta t<1$.
\end{corollary}
\begin{proof}
This follows directly from Theorem \ref{th:expenbm} and Corollary \ref{cor:maincor}.
\end{proof}
To find a comparison theorem which takes into account negative curvature seems harder. We can, however, perform an explicit calculation for the following special case, which compares favourably with our best estimate.
\begin{example}
Suppose that $X(x)$ is a Brownian motion on $\mathbb{H}^{3}_\kappa$ starting at $x$. Then, using formulae \eqref{eq:hkonhyp3} and \eqref{eq:thetaonhyp3}, one can show that
\begin{equation}\label{eq:expmomh3}
\mathbb{E}[e^{\frac{\theta}{2} r^2_x(X_t(x))}] = \left(1-\theta t\right)^{-\frac{3}{2}}\exp\left[-\frac{\theta \kappa t^2}{2(1-\theta t)}\right]
\end{equation}
for all $t>0$ and $\theta \geq 0$ such that $\theta t <1$. Note that the explosion time of the right-hand side of formula \eqref{eq:expmomh3} is independent of $\kappa$.
\end{example}

\subsection{Concentration Inequalities}\label{ss:cncntrtnineq}

If $X(x)$ is a Brownian motion on $\mathbb{R}^{\dimm}$ starting at $x$ then it is easy to see that
\begin{equation}\label{eq:asymrmn3}
\lim_{r\rightarrow \infty}\frac{1}{r^2}\log \mathbb{P}\lbrace X_t(x) \not\in B_r(x)\rbrace = -\frac{1}{2t}
\end{equation}
for all $t>0$. Note that the right-hand side of the asymptotic relation \eqref{eq:asymrmn3} does not depend on the dimension $\dimm$. Returning to the setting of Example \ref{ex:hyp3ex} for a final time, we find another situation where there is a relation of the type \eqref{eq:asymrmn3}.
\begin{example}
Suppose that $X(x)$ is a Brownian motion on $\mathbb{H}^{3}_\kappa$ starting at $x$. Then, by tedious calculation, one can show that
\begin{equation*}
\lim_{r\rightarrow \infty}\frac{1}{r^2}\log \mathbb{P}\lbrace X_t(x) \not\in B_r(x)\rbrace = -\frac{1}{2t}
\end{equation*}
for all $t>0$.
\end{example}
A heat kernel comparison argument would suggest that a relation of this type should hold in general for a Brownian motion $X(x)$ on $M$ but as an inequality, so long as the Ricci curvature is bounded below by a constant. Indeed, it follows from \cite[Theorem~8.62]{MR1715265} that in this case there is the asymptotic estimate
\begin{equation*}
\lim_{r \uparrow \infty}\frac{1}{r^2}\log \mathbb{P}\bigg\lbrace \sup_{s\in \left[0,t\right]} r_x(X_t(x)) \geq r\bigg\rbrace \leq -\frac{1}{2t}.
\end{equation*}
For the general setting considered in this article, we have the following theorem.
\begin{theorem}\label{th:concenthm}
Suppose there exists constants $\nu \geq1$ and $\lambda \in \mathbb{R}$ such that inequality \eqref{eq:triineq} holds and suppose that $X(x)$ is nonexplosive. Then
\begin{equation*}
\lim_{r\rightarrow \infty} \frac{1}{r^2}\log \mathbb{P} \lbrace X_t(x) \notin B_r(N) \rbrace \leq -\frac{1}{2R(t) e^{\lambda t}}
\end{equation*}
for all $t>0$, where $R(t):=(1-e^{-\lambda t})/\lambda$.
\end{theorem}
\begin{proof}
The proof follows a standard argument. In particular, for $\theta \geq 0$ and $r>0$ it follows from Markov's inequality and Theorem \ref{th:expenbm} that
\begin{equation*}
\begin{split}
\mathbb{P} \lbrace X_t(x) \notin B_r(N)\rbrace &= \mathbb{P} \lbrace r_N(X_t(x))\geq r\rbrace\\
&= \mathbb{P} \lbrace e^{\frac{\theta r_N^2(X_t(x))}{2}}\geq e^{\frac{\theta r^2}{2}}\rbrace\\
&\leq e^{-\frac{\theta r^2}{2}}\mathbb{E}\left[e^{\frac{\theta r_N^2(X_t(x))}{2}}\right] \\
&\leq \left(1-\theta R(t) e^{\lambda t}\right)^{-\frac{\nu}{2}} \exp \left[\frac{\theta r_N^2(x) e^{\lambda t}}{2(1-\theta R(t) e^{\lambda t})}-\frac{\theta r^2}{2}\right]
\end{split}
\end{equation*}
so long as $\theta R(t) e^{\lambda t}<1$. If $t>0$ then choosing $\theta =\delta (R(t)e^{\lambda t})^{-1} $ shows that for any $\delta \in \left[0,1\right)$ and $r>0$ we have the estimate
\begin{equation}\label{eq:concentrate}
\mathbb{P} \lbrace X_t(x) \notin B_r(N) \rbrace \leq (1-\delta)^{-\frac{\nu}{2}}\exp\left[\frac{r_N^2(x)\delta}{2R(t)(1-\delta)}-\frac{\delta r^2}{2R(t)e^{\lambda t}}\right]
\end{equation}
from which the theorem follows since $\delta$ can be chosen arbitrarily close to $1$ after taking the limit.
\end{proof}
While Theorem \ref{th:concenthm} is trivial if $M$ is compact, the concentration inequality \eqref{eq:concentrate} is still valid in that setting. In fact, for $r>0$ suppose that $\nu\geq 1$ and $\lambda \geq 0$ are constants such that the inequality
\begin{equation*}
\left(\frac{1}{2}\triangle +b\right) \leq \nu + \lambda r_N^2
\end{equation*}
holds on the tubular neighbourhood $B_r(N)$ (such constants always exist if $N$ is compact, by Corollary \ref{cor:corone}). Assuming that $X(x)$ is non-explosive (which would be the case if $N$ is compact, by Theorem \ref{th:nonexp}) then the methods of this chapter can also be used to estimate quantities involving the process $X(x)$ stopped on the boundary of the tubular neighbourhood. We will not include such calculations here, to avoid extensive repetition, but doing so actually yields the exit time estimate
\begin{equation*}
\mathbb{P}\bigg\lbrace \sup_{s\in \left[0,t\right]} r_N(X_s(x)) \geq r \bigg\rbrace\leq (1-\delta)^{-\frac{\nu}{2}}\exp\left[\frac{r_N^2(x)\delta}{2R(t)(1-\delta)}-\frac{\delta r^2}{2R(t)e^{\lambda t}}\right]
\end{equation*}
for all $t>0$ and $\delta \in (0,1)$, which improves inequality \eqref{eq:concentrate} for the $\lambda \geq 0$ case.

\subsection{Feynman-Kac Estimates}\label{ss:fkineq}

The following two propositions and their corollaries constitute simple applications of Theorems \ref{th:inequality2} and \ref{th:expenbm} and provide bounds on the operator norm of certain Feynman-Kac semigroups, when acting on suitable Banach spaces of functions.
\begin{proposition}\label{prop:FKestpre2}
Suppose there exists constants $\nu \geq2$ and $\lambda \geq 0$ such that inequality \eqref{eq:triineq} holds. Then
\small
\begin{equation*}
\mathbb{E} \left[\mathbf{1}_{\{t< \zeta(x) \}} e^{\theta \int_0^t r_N(X_s(x))ds}\right]\leq 1 +\left( 1+\mathbf{R}(t,\theta t,x)^{-\frac{1}{2}}\right)\left(\mathstrut_1 F_1\left(\frac{\nu}{2},\frac{1}{2}, \mathbf{R}(t,\theta t,x)\right)-1\right) 
\end{equation*}
\normalsize
for all $t,\theta \geq 0$, where $\mathbf{R}$ is defined by \eqref{eq:boldRdefn}.
\end{proposition}
\begin{proof}
Using the stopping times $\tau_{D_i}$ to safely exchange the order of integrals, we see by Jensen's inequality that
\begin{equation*}
\mathbb{E} \left[\mathbf{1}_{\{t< \tau_{D_i} \}}e^{\theta \int_0^t r_N(X_s(x))ds}\right] \leq \frac{1}{t}\int_0^t \mathbb{E} \left[\mathbf{1}_{\{s< \tau_{D_i}\}}  e^{t\theta  r_N(X_s(x))}\right]ds. 
\end{equation*}
The result follows from this by the monotone convergence theorem and Theorem \ref{th:inequality2}, since the right-hand side of inequality \eqref{eq:estfordistexp} is nondecreasing in $t$ (which is evident by the right-hand side of inequality \eqref{eq:rhsasasm} and the fact that $\mathbf{R}(t,\theta,x)$ is nondecreasing in $t$). \qed
\end{proof}
\begin{corollary}\label{cor:Vest2}
Suppose there exists constants $\nu \geq2$ and $\lambda \geq 0$ such that inequality \eqref{eq:triineq} holds and that $V$ is a measurable function on $M$ such that $V\leq C(1+r_N)$ for some constant $C\geq 0$. Then
\footnotesize
\begin{equation*}
\mathbb{E} \left[\mathbf{1}_{\{t< \zeta(x) \}} e^{\int_0^t V(X_s(x))ds}\right] \leq e^{Ct}\left(1 +\left( 1+\mathbf{R}(t,Ct,x)^{-\frac{1}{2}}\right)\left(\mathstrut_1 F_1\left(\frac{\nu}{2},\frac{1}{2}, \mathbf{R}(t,Ct,x)\right)-1\right)\right) 
\end{equation*}
\normalsize
for all $t\geq 0$, where $\mathbf{R}$ is defined by \eqref{eq:boldRdefn}.
\end{corollary}
Using Theorem \ref{th:expenbm} the following proposition and its corollary can be proved in much the same way.
\begin{proposition}\label{prop:FKestpre}
Suppose there exists constants $\nu \geq1$ and $\lambda\in \mathbb{R}$ such that inequality \eqref{eq:triineq} holds. Then
\begin{equation*}
\mathbb{E} \left[\mathbf{1}_{\{t< \zeta(x) \}}e^{\frac{\theta}{2} \int_0^t r_N^2(X_s(x))ds}\right]\leq \left(1-\theta tR(t) e^{\lambda t}\right)^{-\frac{\nu}{2}}\exp\left[\frac{\theta r_N^2(x) te^{\lambda t}}{2(1-\theta tR(t) e^{\lambda t})}\right] 
\end{equation*}
for all $t,\theta \geq 0$ such that $\theta tR(t) e^{\lambda t} <1$.
\end{proposition}
\begin{corollary}\label{cor:Vest}
Suppose there exists constants $\nu \geq1$ and $\lambda \in \mathbb{R}$ such that inequality \eqref{eq:triineq} holds and that $V$ is a measurable function on $M$ such that $V\leq C(1+\frac{1}{2}r_N^2)$ for some constant $C\geq 0$. Then
\begin{equation*}
\mathbb{E} \left[\mathbf{1}_{\{t< \zeta(x) \}}e^{\int_0^t V(X_s(x))ds}\right] \leq \left(1-C tR(t) e^{\lambda t}\right)^{-\frac{\nu}{2}}\exp\left[Ct+\frac{C r_N^2(x) te^{\lambda t}}{2(1-C tR(t) e^{\lambda t})}\right]
\end{equation*}
for all $t\geq 0$ such that $C tR(t) e^{\lambda t} <1$.
\end{corollary}

\subsection{Further Applications}

Applications of the results and methods presented in this article have been explored in \cite{THESIS} and will feature in a subsequent article. In particular, we will show how one can use the Jacobian inequality given by Theorem \ref{th:estimatefinal} and moment estimates for a certain elementary bridge process, related to Brownian motion, to deduce lower bounds and an asymptotic relation for the \textit{integral of the heat kernel over a submanifold}. This object appears naturally in the study of \textit{submanifold bridge processes}, for which the lower bounds imply a gradient estimate sufficient to prove a semimartingale property. This is a new area of study which could lead to future developments related to the geometries of path or loop spaces.

\end{document}